\documentclass[11pt]{amsart}

\usepackage{mathtools}
\usepackage[round]{natbib}
\setcitestyle{numbers,square}
\usepackage{hyperref}

\usepackage{accents}

\newtheorem{theorem}{Theorem}

\newtheorem{lemma}[theorem]{Lemma}
\newtheorem{proposition}[theorem]{Proposition}
\newtheorem{remark}[theorem]{Remark}
\newtheorem{corollary}[theorem]{Corollary}
\newtheorem{example}[theorem]{Example}

\newcommand\bB{\mathbf B}
\newcommand\bF{\mathbf F}
\newcommand\bG{\mathbf G}

\newcommand\bW{\mathbf W}
\newcommand\bX{\mathbf X}
\newcommand\bY{\mathbf Y}
\newcommand\bZ{\mathbf Z}

\newcommand\bx{\boldsymbol x}
\newcommand\by{\boldsymbol y}
\newcommand\bz{\boldsymbol z}

\def\bsigma{\boldsymbol{\sigma}}
\def\bxi{\boldsymbol{\xi}}

\newcommand\cL{\mathcal L}

\newcommand\cP{\mathcal P}

\newcommand\cW{\mathcal W}

\newcommand\HH{\mathbb H}

\newcommand\NN{\mathbb N}
\newcommand\RR{\mathbb R}

\newcommand{\x}{\bx}
\newcommand{\X}{\bX}
\newcommand{\Y}{\bY}

\DeclarePairedDelimiter\abs{\lvert}{\rvert}%

\title[Backward propagation of chaos]{Backward propagation of chaos}
\author{Mathieu Lauri\`ere \& Ludovic Tangpi}

\thanks{Princeton University, ORFE; \\  
lauriere@princeton.edu, ludovic.tangpi@princeton.edu.}	
\date{\today}

\begin{document}

\begin{abstract} 
	This paper develops a theory of propagation of chaos for a system of weakly interacting particles whose \emph{terminal} configuration is fixed as opposed to the initial configuration as customary.
	Such systems are modeled by backward stochastic differential equations.
	Under standard assumptions on the coefficients of the equations, we prove propagation of chaos results and quantitative estimates on the rate of convergence in Wasserstein distance of the empirical measure of the interacting system to the law of a McKean-Vlasov type equation.
	These results are accompanied by non-asymptotic concentration inequalities.
	As an application, we derive rate of convergence results for solutions of second order semilinear partial differential equations to the solution of a partial differential written on an infinite dimensional space.
\end{abstract}

\maketitle


\section{Introduction}
The theory of propagation of chaos takes its origin in the work of M. Kac \cite{Kac} whose initial aim was to investigate particle system approximations of some nonlocal partial differential equations (PDE) arising in thermodynamics.
The intuitive idea is the following: Consider a large number $n$ of (random) particles starting from $n$ given independent and identically distributed random variables and whose respective dynamics interact.
Because there is no deterministic pattern for the starting position of the particles, one says that the \emph{initial configuration is chaotic}.
Kac's insight was that if the interaction between the particles is ``sufficiently weak'' and the particles are ``symmetric'', then as the size of the system increases, there is less and less interaction and in the limit the particles ``become independent''.
That is, the initial chaotic configuration \emph{propagates} over time.
This intuition was put into firm mathematical ground notably by \citet{Mckean67}, \citet{MR1108185} and \citet{Gaertner88} and has generated a rich literature with a variety of fundamental applications.
We refer for instance to \cite{Malrieu03,Lacker18,Bos-Tal,Shkolnikov12,Jabin-Zhang16,Jabir19,Jabin-Wang} for a few recent developments and applications.
In particular, the theory of propagation of chaos has undoubtedly motivated (and benefited from) the more recent and very active theory of mean-field games introduced by~\citet{MR2295621} and~\citet{MR2346927}.

The basic question motivating the present work is to ask whether Kac's intuition carries over to systems of particles with chaotic \emph{terminal} configurations. 
There are numerous such examples, for instance in quantitative finance where different parties independently set investment goals which need to be met at a prescribed future date, but with inter-temporal trading decisions that are correlated.
More precisely, we ask \emph{whether a chaotic terminal configuration will propagate to past configurations} as the size of the system becomes large.
As mentioned above, an important application at the origin of the theory of propagation of chaos is the particle system approximation of some nonlocal PDEs.
We also analyze such an application in the present setting and use the backward propagation of chaos viewpoint to derive a particle system approximation of a semilinear PDE written on an infinite dimensional space (akin to the master equation in the theory of mean-field games). 
The interest here lies in the fact that, being written on a finite dimensional space, the approximating PDEs are much easier to handle analytically.
For instance, well-developed theories of weak solutions and interior estimates for the gradients are available for such equations.
The main idea leading to this approximation result is the probabilistic representation of solutions of some parabolic PDEs, especially due to \citet{ChassagneuxCrisanDelarue_Master}, which allows us to transform the problem of approximating PDE solutions into a purely probabilistic question.

In the present paper, we model \emph{backward particles} by solutions of backward stochastic differential equations (BSDEs) as introduced by \citet{PP90}.
The interaction is through the empirical distribution of the system.
In the main contributions of the paper we derive various convergence results of the $n$-particle system to solutions of McKean-Vlasov BSDEs under classical Lipschitz continuity conditions on the generator and integrability conditions on the terminal value.
The focus is put on deriving explicit, non-asymptotic convergence rates for the empirical measure as well as the processes.
We strengthen our convergence results by deriving concentration inequalities, some of which dimension-free.
All our results are gathered in the next section.
The main result relies on an adaptation of the coupling technique of \citet{MR1108185} and BSDE estimates on the one hand, and arguments from the theory of measure concentration on the other hand, notably results from \citet{T2bsde}.

To the best of our knowledge, only the papers of \citet{Buck-Dje-Li-Peng09} and \citet{Hu-Ren-Yang} touch upon limit results for interacting backward particles.
Both papers consider a particular type of interaction, see Remark~\ref{rem:special-case-linear} and Example~\ref{ex:convolution-interaction} for details.
In \cite{Buck-Dje-Li-Peng09}, a convergence rate for the interacting processes  to the McKean-Vlasov equation is derived; we recover their result by a different argument based on functional inequalities for BSDEs.
In \cite{Hu-Ren-Yang} (where the term ``backward propagation of chaos'' is first coined) a convergence result for the empirical measure of the interacting particles is obtained.
However, nothing is said concerning the rate of convergence.
Another somewhat related article is the work by \citet{Bri-Car-deRaynal-hu19} on the approximation of BSDEs with normal constraints in law.

The ideas and results of the present paper are also connected to the theory of mean field games, which has recently attracted a surge of interest.
In fact, BSDEs of mean-field type arise naturally in optimality conditions for mean field games (MFG) with interactions through the controls, which are sometimes referred to as ``extended MFG'' or ``MFG of controls'' and have been introduced by~\citet{MR3160525}. Such models are particularly relevant in economics and finance, cf. e.g.~\citet{MR3359708,MR3805247}. 
The connection with mean-field BSDEs stems from Pontryagin's maximum principle and has been stressed by~\citet[Section 4.7.1]{MR3752669} and more recently by~\citet{acciaio2018extended}. 
A more extensive discussion on the applications of our results to large population games and mean-field games will be considered elsewhere.

Concerning the approximation problem of PDEs on the Wasserstein space by PDEs on finite dimensional Euclidean spaces, let us mention that a similar question was first analyzed by \citet{carda15} (see also \citet{Las-Lio06,Cardaliaguet17}) based on PDE estimations they derive for the finite dimensional system.
Their results concern the \emph{quasilinear} form of the master equation.
Our contribution here is mainly methodological, as we obtain a convergence result by purely probabilistic techniques.
However, our setting also differs from that of \cite{carda15,Cardaliaguet17} in a number of ways, the most important difference being the type of nonlinearities in the measure argument that we consider.

In the rest of the paper, we dedicate Section \ref{sec:results} to the presentation of the precise setting of the work and its main results.
The proofs are postponed to Section \ref{sec:proofs}.

\section{Setting and main results}
\label{sec:results}
\subsection{Setting and Notation}
\label{subsec:notations}
Let $d,m \in \mathbb{N}$ be fixed. 
Unless otherwise specified, $\RR^d, \RR^m$ and $\RR^{d \times m}$ are endowed with the Euclidean norm denoted by $|\cdot|$ in all cases. Let us denote by $\Omega:= C([0,T],\RR^d)$ the space of continuous functions from $[0,T]$ to $\RR^d$, by $P$ the Wiener measure on $\Omega$ and by $W$ the canonical process given by $W_t(\omega)= \omega(t)$. 
As usual, equalities and inequalities between random variables will be understood to hold up to null sets of the Wiener measure.
It is well-known that $W$ is a $P$-Brownian motion. 
Let $W^1, \dots, W^n$ be $n$ independent copies of $W$ and denote by ${\mathcal{F}}^n:=({\mathcal{F}}_t^n)_{t\in[0,T]}$  the completion of the raw filtration of $W^1, \dots, W^n$. 
Let us equip $\Omega$ with the filtration ${\mathcal{F}}^n$.
We will always use the identification
\begin{equation*}
	W \equiv W^1 \quad \text{and}\quad {\mathcal{F}} \equiv {\mathcal{F}}^1.
\end{equation*}

Given a vector $\x := (x_1, \dots, x_n)\in (\mathbb{R}^m)^n$, denote by 
\begin{equation*}
	L^n(\x) := \frac 1n \sum_{k=1}^n\delta_{x_k}
\end{equation*}
the empirical measure associated to $\x$. 
Then, $L^n(\x) \in {\mathcal{P}}_p(\mathbb{R}^m)$, the set of probability measures on $\mathbb{R}^m$ with finite $p^{th}$ moment.
Let us be given a function $F:[0,T]\times\Omega\times \mathbb{R}^m\times \mathbb{R}^{m\times d}\times {\mathcal{P}}_2(\mathbb{R}^m) \to \mathbb{R}^m$, and a family of ${\mathcal{F}}_T$-measurable i.i.d. random variables $G^1, \dots, G^n$.
We are interested in the asymptotic behavior (as $n$ becomes large) of a family of weakly interacting processes $(Y^{1,n}, \cdots, Y^{n,n})$ evolving backward in time and given by 
\begin{equation}
\label{eq:bsde system}
	Y^{i,n}_t = G^i +\int_t^T F_u(Y^{i,n}_u, Z^{i,i,n}_u, L^n(\Y_u))\,du - \sum_{k=1}^n\int_t^TZ^{i,k,n}_u\,dW^k_u,\quad i = 1,\dots, n,
\end{equation}
where we used the notation $\Y := (Y^{1,n},\dots, Y^{n,n})$. 
Here as well as in the remainder of the article, we assume that for every $(y, z, \mu) \in \mathbb{R}^m\times \mathbb{R}^{m\times d}\times {\mathcal{P}}_2(\mathbb{R}^m)$ the stochastic process $F(\cdot, \cdot, y, z, \mu): (t, \omega) \mapsto F(t, \omega, y, z,\mu)$ is progressively measurable.
In analogy to weakly interacting particles evolving forward in time, in the limit, the above family will be intrinsically linked to the so-called McKean-Vlasov BSDE
\begin{equation}
\label{eq:MkV BSDE}
	Y_t^i = G^i +\int_t^T F_u( Y^i_u, Z^i_u, \cL(Y^i_u))\,du - \int_t^TZ^i_u\,dW^i_u .
\end{equation}
Hereby (and henceforth) $\cL(X)$ denotes the law of the random variable $X$ with respect to the probability measure $P$.
Since under our assumptions on $F$ and $G^i$ the processes $(Y^i)_i$ will be i.i.d., we will often omit the superscript $i$ and simply write $\cL(Y)$ for the law of $Y^i$.

We equip the space ${\mathcal{P}}_p(\mathbb{R}^m)$ with the $p^{\text{th}}$ order Wasserstein distance denoted by ${\mathcal W}_p$ and defined as
\begin{equation*}
	{\mathcal W}_p(\mu, \nu) := \inf\left\{\int_{\mathbb{R}^m\times \mathbb{R}^m}|x-y|^p\,d\pi \right\}^{1/p}
\end{equation*}
where the infimum is over probability measures $\pi$ on $\mathbb{R}^m\times \mathbb{R}^m$ with first and second marginals $\mu$ and $\nu$, respectively.
Given $p \in [1,2]$, we will often consider the condition 
\begin{itemize}
	\item[(Lip$_p$)] The function $F$ is $L_F$-Lipschitz continuous and of linear growth in the sense that there is a constant $L_F \ge 0$ such that, 
	\begin{equation*}
		|F_t(y,z,\mu) - F_t(y',z',\mu')|\le L_F \left(|y-y'| + |z-z'|+ {\mathcal W}_p(\mu, \mu') \right)
	\end{equation*}
	and
	\begin{equation*}
		|F_t(y,z,\mu)| \le L_F\left(1 + |y| + |z| + \Big(\int_{\RR^d}|x|^p\,d\mu \Big)^{1/p} \right)
	\end{equation*}
	for all $t \in [0,T]$, $y,y' \in \mathbb{R}^m$, $z,z'\in \mathbb{R}^{m\times d}$ and $\mu, \mu' \in {\mathcal{P}}_p(\mathbb{R}^m)$.
\end{itemize}
\begin{remark}
\label{rem:existence}
	Note at once that under condition (Lip$_p$), and if $G^i$ has a finite second moment, i.e. $E[|G^i|^2]<\infty$, then the equations \eqref{eq:bsde system} and $\eqref{eq:MkV BSDE}$ admit unique, square integrable solutions. See the beginning of Section \ref{sec:proofs} for details.
\end{remark}
Throughout, we denote by $\Y$ the value process of the solution of \eqref{eq:bsde system} and by $Y$ that of \eqref{eq:MkV BSDE}, say with $i=1$.

Having made precise the probabilistic setting governing the paper, let us now presents its main results.
Most of them pertain to the limiting behavior of $Y^{i,n}$.
As explained in the introduction, we also deduce approximation of parabolic PDEs on the Wasserstein space.
The focus is put on quantitative (i.e. non-asymptotic) estimations of convergence rates.
All proofs are postponed to Section \ref{sec:proofs}.

\subsection{Convergence of empirical distributions}
\label{sec:quant chaos}
We start by showing that the empirical distribution $L^n(\Y_t)$ of the system converges to the law $\cL(Y_t)$ of the McKean-Vlasov BSDE.
\begin{theorem}
\label{thm:mom bound}
	Let $p \in [1,2]$. Assume that $E[|G^i|^k]<\infty$ for some $k$ such that $k>p$ and $k\ge2$, and that $F$ satisfies (Lip$_p$).
	Then it holds that
	\begin{equation}
	\label{eq:mom bound}
		E\left[{\mathcal W}_p^p(L^n(\Y_t),\cL(Y_t)) \right] \le C r_{n,m,q,p}, \qquad \forall t \in [0,T],
	\end{equation}
			with
	\begin{equation}
	\label{eq:def_r-nmqp}
		r_{n,m,q,p}:=\begin{cases}
			n^{-1/2} + n^{-(q-p)/q}, 	&\text{if } p>m/2 \text{ and } q\neq 2p\\
			n^{-1/2}\log(1+n) + n^{-(q-p)/q}, 	&\text{if } p = m/2 \text{ and } q \neq 2p\\
			n^{-p/m} + n^{-(q-p)/q}, 	&\text{if } p\in(0,m/2) \text{ and } q\neq d/(m-p)
		\end{cases}
	\end{equation}
	for all $p<q<k$,
	and for some constant $C$ depending on $T, m, L_F, p,k$ and  $E[|G^i|^k]$. 
\end{theorem}
It is well-known that the Wasserstein topology is much stronger than the weak topology.
Thus, Theorem \ref{thm:mom bound} shows, in particular, that the sequence of (random) measures $(L^n(\Y_t))_n$ converges to the (deterministic) measure $\cL(Y_t)$ in the weak topology.
This can be seen as a type of quantitative law of large numbers.
As a direct application we obtain the following strong law of large numbers for the sequence $Y^{i, n}$.	
\begin{corollary}
\label{corlln}
	Let $p \in [1,2)$. Assume that $E[|G^i|^2]<\infty$ and that $F$ satisfies (Lip$_p$).
	Then we have the $L^1(\Omega, P)$-limit
	\begin{equation*}
		\lim_{n\to \infty}\frac1n \sum_{i=1}^nY^{i,n}_t  = E[Y_t]\quad \text{for every}\quad t.
	\end{equation*}
\end{corollary}
\begin{proof}
 	By the Kantorovich-Rubinstein duality, we have
	\begin{equation*}
	E\left[|\int_{\mathbb{R}^m}f(y)d\,L^n(\Y_t)(y) - \int_{\mathbb{R}^m}f(y)\,d \cL(Y_t)(y)| \right] \le E\left[ {\mathcal W}_1(L^n(\Y_t), \cL(Y_t)) \right] \le Cr_{n,m,q,p}
	\end{equation*}
	for some $p<q<2$ and for every $1$-Lipschitz function $f:\mathbb{R}^m\to \mathbb{R}$.
	In particular, taking $f(x)=x$ yields the result.
\end{proof}

\begin{remark}
\label{rem: r vs m+5}
	Under a stronger integrability condition, namely that $E[|G^i|^k]<\infty$ for some $k>m+5$, the argument of the above theorem allows to obtain the bound
	\begin{equation}
	\label{eq:estim Karandikar}
			E\left[{\mathcal W}^p_p(L^n(\Y_t), \cL(Y_t)) \right]\le Cn^{-p/(m+4)} \quad \text{for all } (t,n)\in[0,T]\times\mathbb{N}
	\end{equation}
	 for some constant $C$ depending only on $T, m, L_F, G$ and $E[|G^i|^{k}]$.
\end{remark}
The estimates \eqref{eq:mom bound} and \eqref{eq:estim Karandikar} are uniform in time in the sense that the convergence rate is time-independent, but the supremum (in $t$) can be taken only outside the expectation on the left hand side.
A stronger uniform estimate can be obtained at the cost of also stronger integrability conditions and a worse convergence rate.
\begin{proposition}
\label{prop:mom bound sup}
	Assume that $E[|G^i|^k]<\infty$ for some $k>m+5$, and that $F$ satisfies (Lip$_p$) for some $p \in [1,2]$.
	Further assume that
	 the solution $(Y,Z)$ of \eqref{eq:MkV BSDE} is such that $\sup_{t \in [0,T]} E[|Z_t|^{2k}]< \infty$.
	Then it holds
	\begin{equation}
	\label{eq:sup mom bound}
		E\Big[\sup_{t\in[0,T]}{\mathcal W}^p_p(L^n(\Y_t), \cL(Y_t)) \Big]\le Cn^{-p/{(m+8)}} \quad \text{for all } n\in \mathbb{N}, \text{ and } p\in [1,2]
	\end{equation}
	for some constant $C$ depending on $T,L_F, k, E[|G^i|^{k}]$ and $\sup_{t \in [0,T]} E[|Z_t|^{2k}]$.
\end{proposition}
\begin{remark}
	The assumption $\sup_tE[|Z_t|^{2k}]< \infty$ is by no means a restrictive one, since it has been shown to hold in many classical cases.
	For instance, when $G^i=G(W_T^i)$ for a bounded and Lipschitz continuous function $G$, and $F_t(\cdot,\cdot,\mu)$ is differentiable for all $(t,\mu)$, then it holds $E[\sup_{t \in [0,T]}|Z_t|^{2k}]<\infty$ for all $k\ge1$, see \cite[Theorem 5.3]{dReis-Imk10}.
	Alternatively, under conditions on the Malliavin differentiability of $G$ and $F$, it can be shown that $Z$ is even bounded, see \cite{Che-Nam,Kup-Luo-Tang18} for details.
	The results of these papers apply for instance when $G$ is Lipschitz continuous on the path space equipped with the supremum norm and $F$ is deterministic.
	In this case, the integrability condition on $G$ also follows.
\end{remark}

\subsection{Concentration estimates}

Given two probability measures $Q_1$ and $Q_2$ on $\Omega$, let us denote the $p^{\text{th}}$ order Wasserstein distance on $\Omega$ equipped with the supremum norm by
\begin{equation*}
	{\mathcal W}_{p,||\cdot||_\infty}(Q_1,Q_2):= \inf\left\{\int_{\Omega\times \Omega}\sup_{t \in [0,T]}|\omega_1(t) - \omega_2(t)|^p\,d\pi(\omega_1, \omega_2) \right\}^{1/p}
\end{equation*}
where the infimum is over probability measures $\pi$ on $\Omega\times \Omega$ with first and second marginals $Q_1$ and $Q_2$, respectively.
The following result gives concentration estimates for the interacting family $\Y$.
We consider concentration for the time $t$ marginal as well as for the law of the entire process.
\begin{theorem}
\label{thm:concentration}
	Let $p \in [1,2]$.
	Assume that $E[|G^i|^k]<\infty$ for some $k>2p$, and that $F$ satisfies (Lip$_p$).
	Then it holds that, for all $ \varepsilon \in (0, \infty)$  and $\varepsilon_{F,T} := \varepsilon / \exp(Te^{L_FT})$,
	\begin{equation}
	\label{eq:conent bound}
		P\left({\mathcal W}_p(L^n(\Y_t), \cL(Y_t)) \ge \varepsilon\right)\le C\big( a_{n, \varepsilon_{F,T}}1_{\{\varepsilon_{F,T}\le 1\}} + b_{n,k,\varepsilon_{F,T}} \big)
	\end{equation}
	with $b_{n,k,\varepsilon} := n(n\varepsilon)^{-(k-\delta)/p}$ and
	\begin{equation*}
		a_{n,\varepsilon} := \begin{cases}
			\exp(-cn\varepsilon^2),	& \text{if } p>m/2\\
			\exp(-c n(\varepsilon/\log(2+1/x))^2),	& \text{if } p=m/2\\
			\exp(-cn\varepsilon^{m/p}), 	& \text{if } p \in (0,m/2)
		\end{cases}
	\end{equation*}
	for three positive constants $\delta \in (0,k)$, $C$ and $c$ depending on $p,m,k$, $T$, $L_F$ and $E[|G^i|^{k}]$.

	Moreover, if the functions $F_t$  and $G^i$ are also Lipschitz continuous as functions on $(\Omega,||\cdot||_\infty)$, that is,
	\begin{align*}
	 	|F_t(\omega,y,z,\mu) - F_t(\omega',y',z',\mu')|&\le L_F\left(||\omega-\omega'||_\infty + |y-y'| + |z-z'|+ {\mathcal W}_p(\mu, \mu') \right)\\
	 	\text{and} \quad |G^i(\omega) - G^i(\omega')|& \le L_G||\omega - \omega'||_\infty,
	 \end{align*}
	then it holds that
	\begin{equation}
	\label{eq:concent 2 proc}
		P\left(\left| {\mathcal W}_{2,\|\cdot\|_\infty}(L^n(\Y), \cL(Y)) - E[{\mathcal W}_{2,\|\cdot\|_\infty}(L^n(\Y), \cL(Y))] \right| \ge\varepsilon \right)\le 2e^{-C\varepsilon^2n}
	\end{equation}
	for a constant $C$ depending only on $L_F, L_G$ and $T$.

	If in addition $F$ does not depend on $z$, then there is $n_0\in \mathbb{N}$ such that for all $n\ge n_0$ we have
	\begin{equation}
	\label{eq:concen path}
		P\left({\mathcal W}_{2,||\cdot||_\infty}(L^n(\Y), \cL(Y) )\ge \varepsilon \right) \le e^{-C \varepsilon^2n}
	\end{equation}
	for some constant $C$ depending on $T, L_F, L_G, m$ and $k$.
\end{theorem}
The proof of Theorem \ref{thm:concentration} relies on quadratic transportation inequalities for BSDEs investigated in \cite{T2bsde} and on standard results from the theory of concentration of measure, see Section \ref{sec:proof conc}.
\subsection{Interacting particles approximation of McKean-Vlasov BSDE}
\label{sec:inter-bsde}
This section is concerned with convergence of the sequence of  stochastic processes $(Y^{i,n}, Z^{i,n})$ to the solutions of the McKean-Vlasov equation.
These results will easily yield quantitative propagation of chaos results and have interesting applications in terms of PDEs.

\begin{theorem}
\label{thm:process conv}
	Assume that $E[|G^1|^k]<\infty$ for some $k>2$, and that $F$ satisfies (Lip$_p$).
	Then it holds that
	\begin{equation}
	\label{eq:process conv}
		E\left[\sup_{t\in [0,T]}|Y^{1,n}_t - Y^1_t |^2 \right]+ E\left[ \int_0^T|Z^{1,1,n}_t - Z_t^1|^2\,dt \right] \le Cr_{n,m,q,2} \quad \text{for all } (t,n)\in[0,T]\times\mathbb{N}
	\end{equation}
	for all $q \in (2,k)$ and for some constant $C$ depending on $T, m, L_F, L_G$ and $E[|Y_t^1|^{k}]$ and $r_{n,m,q,2}$ is defined by~\eqref{eq:def_r-nmqp}. 
\end{theorem}
\begin{remark}
\label{rem:special-case-linear}
	The above result shows that in general, the sequences $(Y^{i,n})$ and $(Z^{i,i,n})$ converge at the same rate as $L^n(\Y_t)$.
	In the special case of particles in ``linear'' interaction, such a convergence result has been analyzed in~\cite{Buck-Dje-Li-Peng09}.
	More precisely, \cite{Buck-Dje-Li-Peng09} considers the case when $\Y=(Y^{1,n},\dots, Y^{n,n})$ solves the system
	\begin{equation}
	\label{eq:lin inter}
		Y^{i,n}_t = G^i + \int_t^T\frac1n \sum_{j=1}^n f_u(Y^{i,n}_u,  Y^{j,n}_u,Z^{i,i,n}_u)\,du - \int_t^TZ^{i,n}_u\,dW_u
	\end{equation}
	where $W$ is a given Brownian motion, and $(G^1,\dots,G^n)$ are functions of the terminal values of a system of interacting (forward) particles.
	In this case, the rate of convergence of the $n$-particle system to the McKean-Vlasov equation can be improved and does not depend on the dimension.
	Interestingly, we can slightly generalize the result of \cite{Buck-Dje-Li-Peng09} using different arguments.
	We consider the system 
	\begin{equation}
	\label{eq:gen lin inter}
		Y^{i,n}_t = G^i + \int_t^T F_u\Big(Y_u^{i,n}, Z^{i,i,n}_u,\frac1n \sum_{j=1}^nf_u(Y^{i,n}_u,  Y^{j,n}_u,Z^{i,n}_u) \Big)\,du - \int_t^T\sum_{j=1}^nZ^{i,j,n}_u\,dW^j_u
	\end{equation}
	that often appears in applications, see e.g.~\cite{MR2352434,MR3489817,MR3325083} for linear-quadratic mean-field models and \cite{Hu-Ren-Yang} for a contract theory problem.
	We obtain the usually optimal rate $1/\sqrt{n}$ for this more general system.
\end{remark}

In fact, consider the McKean-Vlasov equation
\begin{equation}
\label{eq:lin inter McKVl}
	Y^i_t = G^i + \int_t^T F_u\Big(Y^i_u, Z^i_u, \int_{\mathbb{R}^m}f_u(Y^i_u, y, Z^{i}_u)\, d\cL(Y^i_u)(y)\Big)\,du - \int_t^TZ_u^i\,dW^i_u
\end{equation}
and the following Lipschitz continuity and linear growth conditions
\begin{itemize}
	\item[(Lip)] The functions $F:[0,T]\times \Omega\times\mathbb{R}^m\times \mathbb{R}^{m\times d}\times \mathbb{R}^m\to \mathbb{R}^m$ and $f:[0,T]\times\mathbb{R}^m\times\mathbb{R}^m\times \mathbb{R}^{m\times d}\to \mathbb{R}^m$ are respectively $L_F$-Lipschitz and $L_f$-Lipschitz continuous 
	and of linear growth in $(y,z,a)$ and $(y_1,y_2,z)$ uniformly with respect to $(t,\omega)$ and $t$ respectively.
	That is, there are constants $L_F,L_f\ge0$ such that
	\begin{equation*}
		|F_t(\omega,y,z,a)-F_t(\omega',y',z',a')|\le L_F\left(\|\omega - \omega'\|_\infty+|y-y'| + |z-z'| + |a-a'| \right)
	\end{equation*}
	\begin{equation*}
		\qquad|F_t(y,z,a)| \le L_F\left(1 + |y| + |z| +|a| \right)
	\end{equation*}
	and
	\begin{equation*}
		|f_t(y_1,y_2,z)-f_t(y'_1,y_2',z')| \le L_f\left(|y_1-y_1'| +|y_2-y_2'|+ |z-z'| \right)
	\end{equation*}
	\begin{equation*}
		\,\qquad|f_t(y,y_1,z)| \le L_f\left(1 + |y| + |y_1| +|z| \right)
	\end{equation*}
	for all $t \in [0,T]$,  $a,a',y,y',y_1,y_2, y_1', y_2' \in \mathbb{R}^m$, $z,z'\in \mathbb{R}^{m \times d}$.
\end{itemize}
\begin{remark}
\label{rem:lin syst exists}
	Note that under (Lip), if $G^i$ has a second moment, then both equations \eqref{eq:gen lin inter} and \eqref{eq:lin inter McKVl} admits unique, square integrable solutions.
\end{remark}

\begin{proposition}
\label{prop:lin inter}
	Assume that $E[|G^i|^2]<\infty$ and is Lipschitz continuous with respect to the uniform norm on $\Omega$, and that the functions $F$ and $f$ satisfy (Lip).
	The respective solutions $ (Y^{i,n},Z^{i,j,n})$ and $(Y^i,Z^i)$ of the equations \eqref{eq:gen lin inter} and \eqref{eq:lin inter McKVl} satisfy
	\begin{equation}
	\label{eq:mom bound_2}
			E\left[|Y^{1,n}_t - Y_t^1|^2 \right] + E\left[\int_0^T|Z^{1,n}_t- Z_t^1|^2\,dt \right]\le C{n}^{-1} \quad \text{for all } (t,n)\in[0,T]\times\mathbb{N}
	\end{equation}
	for some constant $C$ depending only on $T,L_F,L_f$ and $L_G$.
\end{proposition}
Direct consequences of Theorem \ref{thm:process conv} and Proposition \ref{prop:lin inter} are the following quantitative propagations of chaos. 
\begin{corollary}
\label{cor:prob chaos}
	Put $\theta^{k,n}:=\text{Law}(Y^{1,n}, \dots, Y^{k,n})$	and let $\cL(Y)^{\otimes k}$ be the $k$-fold product of the law $\cL(Y^1)$ of $Y^1$, solution of the McKean-Vlasov BSDE \eqref{eq:MkV BSDE}.
	If $ (Y^{i,n},Z^{i,j,n})$ and $(Y^i,Z^i)$ solve \eqref{eq:bsde system} and \eqref{eq:MkV BSDE} respectively, then under the conditions of Theorem \ref{thm:mom bound}, we have, for all $n \in \mathbb{N}$ and all $k \le n$, that
	\begin{equation}
		\label{eq:estim product}
			\begin{cases}
			{\mathcal W}_{2,||\cdot||_\infty}^2(\cL(Y^{1,n}), \cL(Y^1)) \le Cr_{n,m,q,p}\\ 
			{\mathcal W}_{2,||\cdot||_\infty}^2(\theta^{k,n}, \cL(Y)^{\otimes k}) \le k Cr_{n,m,q,p} 
			\end{cases}
	\end{equation}
	for some constant $C$ depending on $T, L_F, L_G$ and $m$.
	
	If  $ (Y^{i,n},Z^{i,j,n})$ and $(Y^i,Z^i)$ solve \eqref{eq:gen lin inter} and \eqref{eq:lin inter McKVl} respectively, then under the conditions of Proposition \ref{prop:lin inter} we have, for all $(t,n)\in[0,T]\times\mathbb{N}$ and all $k \le n$,  that
		\begin{equation}
		\label{eq:estim product line int}
			\begin{cases}
				{\mathcal W}_2^2(\cL(Y^{1,n}_t), \cL(Y^1_t)) \le Cn^{-1}\\
				{\mathcal W}_2^2(\theta^{k,n}_t, \cL(Y_t)^{\otimes k}) \le kCn^{-1} 
			\end{cases}
	\end{equation}
	for some constant $C$ depending on $T, L_F, L_G, L_f$ and $m$.
\end{corollary}
\begin{proof}
	Since $\cL(Y^1_t) = \cL(Y^i_t)$, it follows by definition that
	\begin{equation*}
		{\mathcal W}^2_{2,||\cdot||_\infty}(\cL(Y^{1,n}), \cL(Y^{1})) \le E\left[\sup_{t \in [0,T]}|Y^{1,n}_t - Y^1_t|^2 \right]
	\end{equation*}
	and
	\begin{equation*}
	 {\mathcal W}_{2,||\cdot||_\infty}^2(\theta^{k,n}, \cL(Y)^{\otimes k}) \le E\left[\sum_{i=1}^k\sup_{t \in [0,T]}|Y^{i,n}_t - Y^i_t|^2 \right].
	\end{equation*}
	Thus, \eqref{eq:estim product} follows by \eqref{eq:process conv}. Similarly \eqref{eq:estim product line int} follows by \eqref{eq:mom bound_2}. 
\end{proof}
\begin{example}[Convolution interaction]
\label{ex:convolution-interaction}
	In relation to a principal-agent problem of mean-field type, \citet{Hu-Ren-Yang} investigated the case of the generator $F_t(y,z,\mu):=  \varphi\ast\mu(y)$ for some function $ \varphi:\mathbb{R}^m\to\mathbb{R}^m$, where the convolution $ \varphi\ast\mu$ is defined as $ \varphi\ast \mu(x):= \int_{\mathbb{R}^m} \varphi(x-y)\,d\mu(y)$.
	This case falls within the scope of Proposition \ref{prop:lin inter} (with $F(t, y, z, a) = a$ and $f(y, y', z) = \varphi(y - y')$), and Corollary \ref{cor:prob chaos} additionally gives a sharp convergence rate. 
\end{example}

\subsection{Finite dimensional approximation of parabolic PDEs on the Wasserstein space}

In this subsection, we assume that $F$ does not depend on $(t,\omega)$.
Given four functions $B :\RR^d\times \cP_2(\RR^d)\to \RR^d$, $\sigma:\RR^d\times \cP_2(\RR^d)\to \RR^{d\times d}$, $G:\RR^d\times \cP_2(\RR^d)\to \RR^m$ and $F:\RR^d\times \RR^m\times\RR^{m\times d}\times \cP_2(\RR^d)\times \cP_2(\RR^m)\to \RR^m$,
we consider the PDE 
\begin{align}
\label{eq:master pde}
\begin{cases}
	\partial_tV(t,x,\mu) + B(x,\nu_1)\partial_xV(t,x,\mu) + \frac12\text{tr}(\partial_{xx}V(t,x,\mu)a(x,\nu_1))\\
	\quad + F(x,V(t,x,\mu), \sigma'(x,\nu_1)\partial_xV(t,x,\mu), \nu_1, \nu_2)
	\\
	 \quad + \int_{\RR^d}\partial_\mu V(t,x,\mu)(y)\cdot B(y,\nu_1)d\mu(y) + \int_{\RR^d}\frac12\text{tr}\left(\partial_y\partial_{\mu}V(t,x,\mu)(y) a(y,\nu_1)\right)\,d\mu(y) = 0 
	 \\
	 V(T, x, \mu) = G(x,\mu)
	\end{cases}
\end{align}
with $ (t,x,\mu) \in [0,T)\times \RR^d \times {\mathcal{P}}_2(\RR^d)$, $a := \sigma\sigma'$, $\nu_1$ and $\nu_2$ are the first and second marginals of the probability measure $\nu$, which itself is the law of $(\xi, V(t, \xi, \mu))$ when $\cL(\xi) = \mu$.
The derivative $$\partial_\mu V(t,x,\mu)(y)$$ denotes the so-called Wasserstein derivative  (also called L-derivative) of the function $V$ in the direction of the probability measure $\mu$, see e.g.~\cite{MR2401600,PLLcollege} or~\cite[Chapter 5]{MR3752669} for details.
The goal of this section is to show that the solution $V$ of the PDE \eqref{eq:master pde}, written on the infinite dimensional space $[0,T]\times \RR^d \times {\mathcal{P}}_2(\RR^d)$ can be approximated by a sequence of solutions of PDEs written on the finite dimensional space $[0,T]\times (\RR^d)^n$.
More precisely, we will be interested in the system of PDEs
\begin{equation}
\label{eq:PDE n system}
	\begin{cases}
		\partial_tv^{i,n}(t,\bx) + B(x_i, L^n(\x))\partial_{x_i}v^{i,n}(t,\x) + \frac12\text{tr}\left( \partial_{x_ix_i}v^{i,n}(t,\x)a(x_i,L^n(\x)) \right)
		\\
		\quad + F\left( x_i,v^{i,n}(t,\x), \partial_{x_i}v^{i,n}(t,\x), L^n(\x)\sigma(x_i,L^n(\x)), \frac1n\sum_{j=1}^n\delta_{v^{j,n}(t,\x)} \right) = 0\quad \text{with } (t,\x) \in [0,T]\times (\RR^d)^n
		\\
		v^{i,n}(T,\x) = G\left(x_i,L^n(\x)\right), \quad \x = (x_1, \dots, x_n)\in (\RR^d)^n
		\\
		i=1,\dots,n.
	\end{cases}
\end{equation}
The following condition is copied almost verbatim from \cite{ChassagneuxCrisanDelarue_Master}.
It guarantees the existence of a unique classical solution $V$ to \eqref{eq:master pde} at least for $T$ small enough. 
\begin{itemize}
	\item[(PDE)] The functions $\sigma, B, F$ and $G$ satisfy the following:
	\begin{itemize}
		\item[(PDE1)] 
			The function $\sigma$ is bounded, and the functions $B, \sigma, F$ and $G$ are three times continuously differentiable in $w = (x,y,z)$ and $\mu$, with bounded and Lipschitz-continuous first and second derivatives (with common bound and Lipschitz constant denoted $L_F$).
		\item[(PDE2)]
			There exist a constant $\alpha\ge0$ and a function $\Phi_\alpha: (L^2(\Omega,{\mathcal{F}}_T, P ;\mathbb{R}^{d+m}))^2 \ni(\chi, \chi') \mapsto \Phi_\alpha(\chi, \chi') \in \mathbb{R}_+$ continuous at any point $(\chi, \chi)$ of the diagonal and such that
	\begin{equation*}
		\Phi_\alpha(\chi,\chi') \le E\left[ \Big(1 + |\chi|^{2\alpha} + |\chi'|^{2\alpha} + ||\chi||_2^{2\alpha}\Big)|\chi - \chi'| \right]^{1/2}
	\end{equation*}
	for all $\chi, \chi \in L^2(\Omega,{\mathcal{F}}_T, P ;\mathbb{R}^{d+m})$ satisfying $\cL(\chi) = \cL(\chi')$.
	Moreover, letting $h= B, \sigma, F$ or $G$, it holds
	\begin{equation*}
		\begin{cases}
			|\partial_wh(w,\cL(\chi)) - \partial_wh(w', \cL(\chi'))  | \le L_F\Big( |w - w'| + \Phi_\alpha(\chi, \chi')\Big)\\
			|\partial_\mu h(w,\chi) - \partial_\mu h(w', \chi')  | \le L_F\Big( |w - w'| + \Phi_\alpha(\chi, \chi')\Big)
		\end{cases}
	\end{equation*}
	for all $w = (x,y,z), w' = (x',y', z')$ and $\chi,\chi' \in L^2(\Omega,{\mathcal{F}}_T, P ;\mathbb{R}^{d+m})$.
	Furthermore, for every $\chi \in L^2(\Omega,{\mathcal{F}}_T, P ;\mathbb{R}^{d+m})$, the family $(\partial_\mu h(w,\chi))_w$ is uniformly integrable.
	\item[(PDE3)]
		Letting $h= B, \sigma, F$ or $G$, the mapping $v\mapsto \partial_\mu h(w, \mu)(v)$ is $L_F$-Lipschitz continuous, the mapping  $(w,v)\mapsto \partial_\mu h(w,\mu)(v)$ is continuously differentiable at any point $(w,v)$ such that $v $ is in the support of $\mu$, $(w, v)\mapsto \partial_v[\partial_\mu h(w,\mu)](v)$ and $(w,v)\mapsto \partial_w[\partial_\mu h(w,\mu)](v)$ are continuous and it holds
	\begin{align*}
		&E\big[ | \partial_w[\partial_\mu h(w, \cL(\chi))](\chi) - \partial_w[\partial_\mu h(w', \cL(\chi'))](\chi') |^2 \big]^{1/2}\\
		&\qquad + E\big[ | \partial_v[\partial_\mu h(w,\cL(\chi))](\chi) - \partial_v[\partial_\mu h(w', \cL(\chi'))](\chi') |^2 \big]^{1/2} \le L_F\Big(|w - w'| + \Phi_\alpha(\chi, \chi')\Big).
	\end{align*}
	\end{itemize}
\end{itemize}
Under the condition (PDE), we then have the announced convergence of $v^{1,n}$ to $V$.
More precisely, we have:

\begin{theorem}
\label{thm:lim PDE}
	Assume that $F$ does not depend on $(t, \omega)$ and that the condition (PDE) is satisfied.
	There exists $c$ such that if $T \le c$, then the following holds true. 
	The sequence $(v^{1,n})_n$ converges to $V$ in the sense that for every i.i.d. sequence $(\xi_i)_{i\in \mathbb{N}}$ in $L^k(\Omega,{\mathcal{F}}_t, P; \RR^d)$ for some $k>4$  and every $\mu \in {\mathcal{P}}_2(\RR^d)$ with $\cL(\xi_1) = \mu$, it holds that
	\begin{equation}
	\label{eq:PDE linear inter}
		E\Big[\sup_{t \in [0,T]} | v^{1,n}(t,\xi_1, \dots,\xi_n) - V(t,\xi_1,\mu) |^2 \Big]\le C_{L_F, T,k} \, \varepsilon_n 
	\end{equation}
		where $\varepsilon_n$ is defined as
	\begin{equation}
	\label{eq:def-epsilon-CD}
		\varepsilon_n
		=
		\begin{cases}
			n^{-1/2}, & \hbox{ if } d < 4,\\
			n^{-1/2} \log(n), & \hbox{ if } d = 4,\\
			n^{-2/d}, & \hbox{ if } d > 4
		\end{cases}
	\end{equation}
	 and $C_{L_F, T,k}$ is a constant depending on $L_F,T$ and $E[|\xi_1|^k]$.
	
	Moreover, for every $n \in \NN$ and every $t\in [0,T]$ it holds that 
	\begin{equation}
	\label{eq:PDE linear inter empi}
		 E\Big[| v^{i,n}(t,\xi_1, \dots, \xi_n) - V(t, \xi_i, L^n(\boldsymbol\xi)) |^2 \Big] \le C_{L_F, T}(\varepsilon_n + r_{n,d,k,2})
	\end{equation}
	with $\boldsymbol \xi:= (\xi_1, \dots, \xi_n)$, where $C_{L_F, T}$ depends on the Lipschitz constant $L_F$ of $B,F$ and $G$ and on $T$, and $r_{n,d,k,2}$ is defined by~\eqref{eq:def_r-nmqp}.
\end{theorem}


\section{Proofs}
\label{sec:proofs}
In this final section we give detailed proofs of the results presented above. 
We start by justifying well-posedness of \eqref{eq:MkV BSDE} and \eqref{eq:bsde system}.
For simplicity of notation, we will put
\begin{equation*}
	Z^{i,n} := Z^{i,i,n}
\end{equation*}
whenever this does not cause confusions.
\begin{proof}[Proof of Remark \ref{rem:existence}]
	By \cite[Theorem 4.23]{MR3752669}, the equation \eqref{eq:MkV BSDE} admits a unique solution $(Y,Z)$ solution in the space  ${\mathcal S}^2(\RR^m) \times {\mathcal H}^2(\mathbb{R}^{m \times d})$, where we use the notation: for each integer $k \ge 1$,
	$$
		{\mathcal S}^2(\mathbb{R}^{k}) := \left\{ Y \in \HH^{0}(\mathbb{R}^{k}) \,\Big|\, E \sup_{0 \le t \le T} |Y_t|^2 < +\infty \right\}\quad\text{and}\quad {\mathcal H}^2(\mathbb{R}^{k}) := \left\{ Z \in \HH^{0}(\mathbb{R}^{k}) \,\Big|\, E \int_0^T |Z_t|^2 dt < +\infty \right\},
	$$
	with ${\mathcal H}^{0}(\mathbb{R}^{k})$ being the space of all $\RR^k$-valued progressively measurable processes.

	Moreover, we can apply \cite[Theorem 2.1]{karoui01} to the system \eqref{eq:bsde system} to justify the existence and uniqueness of a solution.
	To that end, it helps to write it in the more  compact form
	\begin{equation}
	\label{eq:bsde system_vector}
		\bY_t = \bG +\int_t^T \bF_u(\bY_u, \bZ_u)\,du - \int_t^T \bZ_u \, d \bW_u, 
	\end{equation}
	where $\bG(\boldsymbol\omega) := (G^1(\omega^1), \dots, G^n(\omega^n))$, $\boldsymbol\omega := (\omega^1,\dots, \omega^n),$
	$$
		\bY := (Y^{i,n})_{i=1,\dots,n}, \qquad  \bZ^{i,n}:= (Z^{i,1,n}, \dots, Z^{i,n,n}),  \quad 
		\qquad \bW = (W^i)_{i=1,\dots,n}
	$$
	 $\bZ:= \mathrm{diag}(\bZ^{i,n})_{i=1,\dots,n}$
	and  $\bF : [0,T]\times \Omega^n \times (\mathbb{R}^m)^n\times (\mathbb{R}^{m\times d})^n \to (\mathbb{R}^m)^{n}$ is defined by
	$$
		\bF(t,\boldsymbol \omega, \by, \bz) = (F_t(\omega^i,y^{i}, z^{i,i}, L^n(\by)))_{i=1,\dots,n}
	$$
	for  $t \in [0,T], \by = (y^1, \dots, y^n) \in (\mathbb{R}^m)^n, \bz = (z^1, \dots, z^n) \in (\mathbb{R}^{m\times d})^n$ and $\boldsymbol \omega:= (\omega^1, \dots, \omega^n) \in \Omega^n$.
	Now it suffices to check that the function $\bF$ is $(8L_F)$-Lipschitz continuous (the Lipschitz constant does not depend on $n$).
	We do this for the reader's convenience.
	For every $t \in [0,T], \by_1 ,\by_2 \in (\mathbb{R}^m)^n, \bz_1,\bz_2  \in (\mathbb{R}^{m\times d})^n$ it holds that
	\begin{align*}
		| \bF(t, \by_1, \bz_1) - \bF(t, \by_2, \bz_2) |^2
		&= \sum_{i=1}^n | F_t(y^{i}_1, z^{i,i}_1, L^n(\by_1)) - F_t(y^{i}_2, z^{i,i}_2, L^n(\by_2)) |^2
		\\
		&\le L_F \sum_{i=1}^n \left( |y^{i}_1 - y^{i}_2| + |z^{i,i}_1 - z^{i,i}_2| +  \cW_p(L^n(\by_1), L^n(\by_2)) \right)^2
		\\
		&\le L_F \sum_{i=1}^n \left( |y^{i}_1 - y^{i}_2| + |z^{i,i}_1 - z^{i,i}_2| +  \cW_2(L^n(\by_1), L^n(\by_2)) \right)^2
		\\
		&\le 4 L_F \sum_{i=1}^n \left( |y^{i}_1 - y^{i}_2|^2 + |z^{i,i}_1 - z^{i,i}_2|^2 +  \frac{1}{n}  \sum_{j=1}^n |y^{j}_1 - y^j_2|^2 \right)
		\\
		&\le 8 L_F \left( |\by_1 - \by_2|^2 + |\bz_1 - \bz_2|^2 \right).
	\end{align*}
	To derive these inequalities, we successively used assumption (Lip$_p$), the fact that $\cW_p(\mu,\nu) \le \cW_2(\mu,\nu)$ for any $p \in [1,2]$, and the fact that $\cW_2$, for two $n$-sample empirical distributions $L^n(\by_1), L^n(\by_2)$ is given by
	$$
		\cW_2(L^n(\by_1), L^n(\by_2))
		=
		\min_\sigma \left( \frac{1}{n} \sum_{i=1}^n |y^{\sigma(i)}_1 - y^i_2|^2 \right)^{1/2},
	$$
	where the minimum is over permutations $\sigma$ of $\{1,\dots,n\}$, see e.g. \cite[Lemma 5.1.7]{CardLect18}.
\end{proof}

\subsection{Proofs for Subsection \ref{sec:quant chaos}}
We begin with two moment estimates for the solution of the McKean-Vlasov BSDE.
Given a square integrable progressive process $q$, we will denote by
\begin{equation*}
	{\mathcal E}_{s,t}(q\cdot W) := \exp\left(\int_s^tq_u\,dW_u - \frac12\int_s^t|q_u|^2\,du \right)
\end{equation*}
the stochastic exponential of $q$, and for every measure $\mu \in {\mathcal{P}}_p(\RR^m)$, we let 
\begin{equation*}
	M_p(\mu):= \int_{\mathbb{R}^m}|x|^p\,d\mu
\end{equation*}
be the $p^{th}$-moment of $\mu$.
\begin{lemma}
\label{lem: moment d+5}
	Let $k\ge2$. Assume $E[|G|^{k}]<\infty$
	and that $F$ satisfies (Lip$_p$) for some $p \in [1,2]$.
	Then, for all $q\ge0$, if $q=2$ or $q<k$, the solution $(Y,Z)$ of \eqref{eq:MkV BSDE} satisfies
	\begin{equation*}
		\sup_tE[|Y_t|^q]<\infty.
	\end{equation*}
\end{lemma}
\begin{proof}
	If $q = 2$, there is nothing to prove because
	the result follows from Remark \ref{rem:existence}. 
	Let us assume $q<k$.
	Since $F$ is Lipschitz continuous in the $z$-variable, it is almost everywhere differentiable. 
	Thus, it follows from the mean-value theorem that
	\begin{align*}
		Y_t &= G + \int_t^T \Big( F_u(Y_u, Z_u, \cL(Y_u)) - F_u(Y_u,0,\cL(Y_u)) + F_u(Y_u,0,\cL(Y_u)) \Big)\,du - \int_t^TZ_u\,dW_u\\
		&= G + \int_t^T \left( \int_0^1\partial_z F_u(Y_u, \lambda Z_u,\cL(Y_u))\,d\lambda Z_u + F_u(Y_u, 0, \cL(Y_u)) \right)\,du - \int_t^TZ_u\,dW_u\\
		&= E_Q\left[G + \int_t^T F_u(Y_u,0,\cL(Y_u))\,du \mid {\mathcal{F}}_t  \right]
	\end{align*}
	where we used Girsanov's theorem, with $Q$ being the probability measure with density $dQ/dP:= {\mathcal E}_{0,T}(\gamma\cdot W)$ and
	\begin{equation*}
	 	\gamma_u:= \int_0^1\partial_z F_u(Y_u, \lambda Z_u,\cL(Y_u))\,d\lambda.
	 \end{equation*}
	 Hence, using the linear growth of $F$ we obtain 
	  \begin{align*}
	 	|Y_t|^q \le C_{F,k,q}E_Q\left[|G|^q +  \int_t^T\Big(1 + |Y_u|^q+M_2^{q/2}(\cL(Y_u)) \Big)\,du \mid {\mathcal{F}}_t \right]
	 \end{align*}
	 for a constant $C_{F,k,q}$ depending only on $k$ and $F$.
	 Thus, by Gronwall's inequality we have
	 \begin{align*}
	 	|Y_t|^q 
		&\le e^{C_{F,k,q}T}C_{F,k,q}E_Q\left[ |G|^q + \sup_{u\in[0,T]}M_2(\cL(Y_u))^{q/2} + T\mid {\mathcal{F}}_t \right]
		\\
	 			&\le e^{C_{F,k,q}T}C_{F,k,q}E\left[{\mathcal E}_{t,T}(\gamma\cdot W)^{k/(k-q)}\mid {\mathcal{F}}_t\right]^{(k-q)/k} E\left[|G|^{k}\mid {\mathcal{F}}_t \right]^{q/k}
				\\
	 			&\quad + C_{F,k,q}(T + \sup_{u\in[0,T]}M_2(\cL(Y_u))^{q/2}).
	 \end{align*}
	 Note that, since $\gamma$ is a bounded process, the random variable ${\mathcal E}_{t,T}(\gamma\cdot W)$ has moments of all orders.
	 Furthermore, $\sup_{u\in[0,T]}M_2(\cL(Y_u))<\infty$ since $Y\in {\mathcal S}^2(\RR^m)$, see  Remark \ref{rem:existence}. 
	 Therefore, taking expectation on both sides and applying again H\"older's inequality concludes the argument. 
\end{proof}

\begin{lemma}
\label{lem:estimation difference}
	Let $k \ge 3$. Assume that $E[|G|^{k'}]<\infty$ for some $k'>2k$ and that $F$ satisfies (Lip$_p$) for some $p \in [1,2]$.
	Assume that the solution $(Y,Z)$ of \eqref{eq:MkV BSDE} is such that $\sup_tE[|Z_t|^{2k}]< \infty$. 
	Let $0 \le t_{1} \le t_{2} \le T$ be such that $t_{2}-t_{1} \le1$. Then we have
	\begin{itemize}
		\item[(i)]  
		$E[|Y_t- Y_s|^k|Y_s - Y_r|^k]\le C|t-r|^2$ for some $C\ge0$ and for all $t_{1}\le r< s<t\le t_{2}$.
		\item[(ii)] 
		$E[|Y_t - Y_s|^k] \le C|t-s|$ for some $C\ge0$ and for every $t_{1}\le s\le t\le t_{2} $.
	\end{itemize}
\end{lemma}
\begin{proof}
	Let us start with the proof of (i).
	A direct estimation and repeated applications of H\"older's inequality yield 
	\begin{align*}
		&E\left[|Y_t - Y_s|^k|Y_s-Y_r|^k\right]\\
		 	&\le 2^{2k-2}(s-r)^{(k-1)}(t-s)^{(k-1)}E\left[\int_r^s| F_u(Y_u, Z_u, \cL(Y_u))|^k\,du\int_s^t| F_u(Y_u, Z_u, \cL(Y_u))|^k\,du \right]\\
		 	&\quad + 2^{2k-2}(s-r)^{(k-1)}E\left[\int_r^s| F_u(Y_u, Z_u, \cL(Y_u))|^{k}\,du\left|\int_s^tZ_u\,dW_u\right|^{k} \right]\\
		 	&\quad + 2^{2k-2}(t-s)^{(k-1)}E\left[\int_s^t| F_u(Y_u, Z_u, \cL(Y_u))|^{k}\,du\left|\int_r^sZ_u\,dW_u\right|^{k} \right]\\
		 	&\quad
		 	 + 2^{2k-2} E\left[\left|\int_r^sZ_u\,dW_u\right|^{2k}\right]^{1/2}E\left[\left|\int_s^tZ_u\,dW_u\right|^{2k} \right]^{1/2}.
	\end{align*}
	Now, recall that by Lemma \ref{lem: moment d+5}, it holds $\sup_{t\in [0,T]}E[|Y_t|^{2k}]<\infty$, $\sup_{t\in[0,T]}M_2(\cL(Y_t))<\infty$ and by assumption, that $\sup_tE[|Z_t|^{2k}]< \infty$.
	Thus, by the linear growth condition on $F$ and Burkholder-Davis-Gundy inequality, we have 
	\begin{align*}
		&E\left[|Y_t - Y_s|^k|Y_s-Y_r|^k\right]\\
		&\le 2^{2k}L_FT(t-r)^{2(k-1)} E\left[\int_0^T \left(1 + |Y_u|^{2k} + |Z_u|^{2k}+M_2^{2k}(\cL(Y_u)) \right)\,du\right]\\
		&\quad + 2^{2k}L_F(t-r)^{(k-1)}E\left[\int_r^s\left(1 + |Y_u|^{2k}+|Z_u|^{2k}+M_2^{2k}(\cL(Y_u)) \right) \right]^{1/2}E\left[\left(\int_s^t|Z_u|^2\,du \right)^{k} \right]^{1/2}\\
		&\quad + 2^{2k}L_F(t-r)^{(k-1)}E\left[\int_s^t \left(1 + |Y_u|^{2k}+|Z_u|^{2k}+M_2^{2k}(\cL(Y_u)) \right) \right]^{1/2}E\left[\left(\int_r^s|Z_u|^2\,du \right)^{k} \right]^{1/2}\\
		&\quad + 2^{2k} E\left[\left(\int_r^s|Z_u|^2\,du \right)^{k} \right]^{1/2}E\left[\left(\int_s^t|Z_u|^2\,du \right)^{k} \right]^{1/2}\\
		&\le 2^{2k}C_{F,T} (t-r)^{2(k-1)} + 2^{2k}C_{F,T}(t-r)^{(k-1)}(s-t)^{k-1} + 2^{2k}C_{F,T}(t-r)^{(k-1)}(t-s)^{k-1}\\
		&\quad + C_{F,T}2^{2k}\left[ (s-r)^{k-1}(t-s)^{k-1} \right]^{1/2}
		\\
		&\le C_{F,T}(t-r)^{k-1}.
	\end{align*}
	Since $k\ge3$ and $t-r\le 1$, we can conclude from the above that
	\begin{equation*}
		E\left[|Y_t - Y_s|^k|Y_s-Y_r|^k\right] \le C_{k,T,F} |t-r|^2,
	\end{equation*}
	where $C_{k,T,F} $ is a constant depending on $k$, $T$, $L_F$ and the $2k^{th}$ moments of $Y$ and $Z$. 
	This proves the first claim.

	Let us turn to the proof of the second claim, which is similar (and simpler).
	In fact, arguing as above we get
	\begin{align*}
		E\left[|Y_t-Y_s|^k\right] &\le 2^{k-1}E\left[ \left|\int_s^t F_u(Y_u, Z_u, \cL(Y_u))\,du\right|^k + \left| \int_s^tZ_u\,dW_u\right|^k \right]\\
		&\le 2^k(t-s)^{(k-1)}E\left[ \int_s^t| F_u(Y_u, Z_u, \cL(Y_u))|^k\,du\right] + 2^kE\left[\left(\int_s^t|Z_u|^2\,du \right)^{k/2} \right]\\
		&\le 2^kL_F(t-s)^{(k-1)} + 2^kC (t-s)^{(k-1)/2} 
		\le C_{F, T}|t-s|, 
	\end{align*}
	where we used the facts that $t-s\le 1$ and $k \ge 3$.

	In the case that $F$ is bounded in $z$ the argument is exactly the same. 
	Since the terms $|Z_t|^{2k}$ will not appear in the estimates we can conclude without assumptions on the moments of $Z$. 
\end{proof}

Next, we will adapt to BSDEs a well-known coupling technique that will allow to use some known quantitative bounds for i.i.d. samples in our interacting particles case.
This coupling technique, which probably originated from the work of \citet{MR1108185}, is by now standard in SDE theory, see e.g. \cite{Del-Lac-Ram_Concent,carda15} for recent references.
Hence, let $(Y,Z)$ be the solution of the McKean-Vlasov BSDE~\eqref{eq:MkV BSDE}.
Let $(\tilde Y^1,\tilde Z^1), \dots (\tilde Y^n, \tilde Z^n)$ be i.i.d. copies of $(Y,Z)$ such that for each $i$, $(\tilde Y^i,\tilde Z^i)$ solves the equation
\begin{equation}
\label{eq:def_tildeYi}
	\tilde Y_t^i = G^i +\int_t^T F_u(\tilde Y_u^i,\tilde Z_u^i, \cL(Y_u))\,du - \int_t^T\tilde Z_u^i\,dW_u^i. 
\end{equation} 
Such copies can be found because the McKean-Vlasov BSDE has a unique solution, and thus we have uniqueness in law. We let $\tilde \Y = (\tilde Y^1, \dots, \tilde Y^n)$. 
\begin{lemma}
\label{lem: fundamental lemma}
	Let $p \in [0,1]$. Assume that $E[|G^1|^2]<\infty$ and that $F$ satisfies (Lip$_p$).
	Then it holds that
	\begin{equation}
	\label{eq:fundamental inequality}
		{\mathcal W}_p(L^n(\Y_t), \cL(Y_t)) \le \exp(Te^{L_FT}) {\mathcal W}_p(L^n(\tilde \Y_t), \cL(Y_t))\quad P\text{-a.s}.
	\end{equation}
	$\text{for all } (t,n) \in [0,T]\times\mathbb{N}$.
\end{lemma}
\begin{proof}
	Let $i \in \{ 1, \dots, n\}$ be fixed.
	It follows by the mean-value theorem that 
	\begin{align*}
		\tilde Y^i_t - Y^{i,n}_t &= \int_t^T \left[ F_u(\tilde Y^i_u,\tilde Z^i_u, \cL(Y_u)) - F_u(Y^{i,n}_u, Z^{i,i,n}_u, L^n(\Y_u)) \right]\,du - \sum_{j=1}^n\int_t^T \left(\delta_{ij}\tilde Z^i_u - Z^{i,j,n}_u \right) \,dW_u^j \\
		&=\int_t^T \left[ \alpha^i_u + \beta^i_u + \int_0^1\partial_z F_u( Y^{i,n}_u,  Z^{i,i,n}_u + \lambda (\tilde Z^i_u -Z^{i,i,n}_u), \cL(Y_u)) \,d\lambda(\tilde Z^i_u - Z^{i,i,n}_u) \right] \,du\\
		&\quad - \sum_{j=1}^n\int_t^T\left(\delta_{ij}\tilde Z^i_u - Z^{i,j,n}_u \right) \,dW_u^j
	\end{align*}
	with $\delta_{ij}=1$ if $i=j$ and $\delta_{ij}=0$ otherwise, $\alpha^i_u:=  F_u(\tilde Y^i_u,\tilde Z^i_u, \cL(Y_u)) - F_u(Y^{i,n}_u,\tilde Z^i_u, \cL(Y_u)) $ and $\beta^i_u:=  F_u( Y^{i,n}_u, Z^{i,i,n}_u,\cL(Y_u)) -  F_u( Y^{i,n}_u, Z^{i,i,n}_u, L^n(\Y_u))$.
	Note that since $F$ is Lipschitz continuous, the derivative $\partial_zF$ can be defined almost everywhere, and is bounded.
	Thus, the density process 
	\begin{equation*}
		{\mathcal E}_{0,t}(\gamma\cdot W^i) \quad \text{with}\quad \gamma_u:=\int_0^1\partial_z F_u( Y^{i,n}_u,  Z^{i,i,n}_u + \lambda (\tilde Z^i_u-Z^{i,n}_u), \cL(Y_u)) \,d\lambda
	\end{equation*}
	defines an equivalent probability measure $Q$.
	Due to Girsanov's theorem and square integrability of $Z^{i,j,n}$ and $\tilde Z^i$, taking the expectation above with respect to $Q$ yields
	\begin{align}
	\label{eq:alpha beta estimate}
		\tilde Y^i_t - Y_t^{i,n} = E_Q\left[\int_t^T (\alpha^i_u + \beta^i_u ) \,du \mid {\mathcal{F}}_t^n \right].
	\end{align}
	Again by Lipschitz continuity of $F$, it holds that
	\begin{equation*}
		|\alpha^i_u| + |\beta^i_u| \le L_F\left(|\tilde Y_u^i - Y^{i,n}_u| + {\mathcal W}_p(L^n(\Y_u), \cL(Y_u)) \right)	
	\end{equation*}	
	so that by Gronwall's inequality, we have
	\begin{equation*}
		|\tilde Y^i_t - Y_t^{i,n}| \le e^{L_FT}E_Q\left[\int_0^T{\mathcal W}_p(L^n(\Y_u), \cL(Y_u))\,du \mid {\mathcal{F}}_t^n \right].
	\end{equation*}
	Hence, using the definition of the $p^{th}$-Wasserstein distance, we obtain the estimate
	\begin{align*}
		{\mathcal W}_p(L^n(\Y_t), L^n(\tilde \Y_t))\le \left( \frac 1n\sum_{i=1}^n|\tilde Y^i_t - Y^i_t|^p\right)^{1/p} \le e^{L_FT}E_Q\left[\int_0^T{\mathcal W}_p(L^n(\Y_u), \cL(Y_u))\,du \mid {\mathcal{F}}_t^n \right].
	\end{align*}
	Now, combine this with the triangle inequality to obtain
	\begin{align}
	\nonumber
		{\mathcal W}_p(L^n(\Y_t),\cL(Y_t)) &\le {\mathcal W}_p(L^n(\Y_t), L^n(\tilde \Y_t)) + {\mathcal W}_p( L^n(\tilde \Y_t), \cL(Y_t)) \\
	\label{eq: lem fund before gronwall}
		&\le e^{L_FT}E_Q\left[\int_0^T{\mathcal W}_p(L^n(\Y_u),\cL(Y_u))\,du \mid {\mathcal{F}}_t^n \right] + {\mathcal W}_p(L^n(\tilde \Y_t), \cL(Y_t)).
	\end{align}
	Applying again Gronwall's inequality yields the desired result.
\end{proof}
With the proofs of the above lemmas aside, we are ready to prove quantitative estimations for the convergence of the empirical measure of $\Y_t$ to the law $\cL(Y_t)$ of the McKean-Vlasov BSDE.

\begin{proof}[Proof of Theorem \ref{thm:mom bound} and Proposition \ref{prop:mom bound sup}]
The proofs begin with Lemma \ref{lem: fundamental lemma}.
In fact this lemma implies that
\begin{equation}
\label{eq:just bef estim}
	E\left[ {\mathcal W}_p^p(L^n(\Y_t), \cL(Y_t))\right] \le \exp(Te^{L_FT})E\left[{\mathcal W}_p^p(L^n(\tilde \Y_t), \cL(Y_t)) \right]\quad \text{for all } (t,n) \in [0,T] \times \mathbb{N}.
\end{equation}
Since $E[|G^i|^k]<\infty $ with $k>p$, we have by Lemma \ref{lem: moment d+5} that $\sup_{t \in [0,T]} E[|Y_t|^q]<\infty$ for $q\in (p,k)$. Thus, it follows by \cite[Theorem 1]{Fou-Gui15} that
\begin{equation*}
	E\left[{\mathcal W}_p^p(L^n(\tilde \Y_t), \cL(Y_t)) \right]\le C r_{n,m,k,p}
\end{equation*}
for a constant $C$ depending on $L_F, T, m, p$ and $k$.
Therefore, the estimate \eqref{eq:mom bound} is obtained due to \eqref{eq:just bef estim}.%

To get the estimate \eqref{eq:sup mom bound}, let $0=t_0 \le t_1 \le \dots \le t_N =T$ be a partition of $[0,T]$ in $N$ intervals of length $t_{j+1} - t_j\le 1$ (if $T\le 1$ we simply take $N=1$).
Considering the decomposition
\begin{equation*}
	{\mathcal W}_p^p(L^n(\Y_t), \cL(Y_t)) = \mathcal{W}_p^p(L^n(\Y_0), \cL(Y_0)) \mathbf{1}_{\{0\}}(t) +  \sum_{j=0}^{N-1}{\mathcal W}_p^p(L^n(\Y_t), \cL(Y_t))1_{(t_j, t_{j+1}]}(t)\quad \text{for all } t \in [0,T],
\end{equation*}
we have
\begin{equation}
\label{eq:estim sum}
	E\Big[\sup_{t \in [0,T]}{\mathcal W}_p^p(L^n(\Y_t), \cL(Y_t))\Big] 
	\le 
	\sum_{j=0}^{N-1} E\Big[\sup_{t \in [t_j,t_{j+1}]}{\mathcal W}_p^p(L^n(\Y_t), \cL(Y_t)) \Big].
\end{equation}
From Lemma \ref{lem: fundamental lemma} we have

\begin{equation*}
	E\Big[\sup_{t \in [t_j, t_{j+1}]}{\mathcal W}_p^p(L^n(\Y_t), \cL(Y_t))\Big] 
	\le 
	\exp(p T e^{L_FT}) E\left[\sup_{t \in [t_j, t_{j+1}]}{\mathcal W}_p^p(L^n(\tilde \Y_t), \cL(Y_t))\right] \quad \text{for all } n \in \mathbb{N},\,\, j\le N-1.
\end{equation*}
Since $p \in [1,2]$, it follows by Jensen's inequality and the inequality ${\mathcal W}_p \le {\mathcal W}_2$ that for all $n \in \mathbb{N}$,
\begin{align}
\label{eq:sup mom almost bound}
	E\left[\sup_{t \in [t_{j},t_{j+1}]}{\mathcal W}_p^p(L^n(\Y_t), \cL(Y_t))\right]^{2/p} \le e^{2L_FT}E\left[\sup_{t \in [t_j,t_{j+1}]}{\mathcal W}_2^2(L^n(\tilde \Y_t), \cL(Y_t))\right]
\end{align}
Since $t_{j+1} - t_j\le 1$, Lemma \ref{lem:estimation difference} applies, in view of which it follows from Lemma \ref{lem: moment d+5} and \cite[Theorem 1.3]{Hor-Kar94} that
\begin{equation*}
	E\Big[\sup_{t \in [t_{j}, t_{j+1}]}{\mathcal W}_2^2(L^n(\tilde \Y_t), \cL(Y_t))\Big] \le C_{G,F,k,m}n^{-2/(m+8)}.
\end{equation*}
Therefore, we deduce from \eqref{eq:sup mom almost bound} and \eqref{eq:estim sum} that
\begin{equation*}
	E\Big[\sup_{t \in [0,T]}{\mathcal W}_p^p(L^n(\Y_t), \cL(Y_t))\Big] \le Ne^{pL_FT}C_{G,F,k,m} n^{-p/(m+8)} .
\end{equation*}
This concludes the proof, since $N$ can be chosen less than $T+1$.
\end{proof}
\begin{proof}[Proof of Remark \ref{rem: r vs m+5}]
	The proof is the same as that of the estimate \eqref{eq:mom bound} with application of \cite[Theorem 1.2]{Hor-Kar94} instead of \cite[Theorem 1]{Fou-Gui15}.
\end{proof}

\subsection{Proof of Theorem \ref{thm:concentration}}
\label{sec:proof conc}
The proofs of Theorem \ref{thm:concentration} and Proposition \ref{prop:lin inter} partially rely on functional inequalities that we now recall for the reader's convenience.
See however \cite[Chapters 21 $\&$ 22]{Vil2} for further details.

Let $\cW_{2,\delta}$ denote the Wasserstein distance of order 2 with respect to a distance $\delta$ on a Polish space $E$.
A probability measure $\mu \in \cP(E)$ is said to satisfy Talagrand's $T_2$ inequality with constant $C$ if
\begin{equation*}
	\cW_{2,\delta}(\mu,\nu) \le \sqrt{CH(\nu|\mu)} \quad \text{for every probability measure } \nu,
\end{equation*}
where $H$ is the Kullback-Leibler divergence defined as
\begin{equation*}
 	H(\nu|\mu) := \begin{cases}
 		\int \log(\frac{d\nu}{d\mu})d\nu, \quad \text{if } \nu \ll \mu\\
 		+\infty, \quad \text{otherwise},
 	\end{cases}
 \end{equation*}
 with the convention $E[X]:=\infty$ whenever $E[X^+] = \infty$.
Below, we will exploit the efficiency of Talagrand's inequality in deriving concentration inequalities, but also the fact that it implies other functional inequalities, notably the $T_1$ inequality
\begin{equation*}
	\cW_{1,\delta}(\mu,\nu) \le \sqrt{CH(\nu|\mu)} \quad \text{for every probability measure } \nu,
\end{equation*}
and Poincar\'e's inequality
\begin{equation*}
 	\mathrm{Var}(f) \le C\int_E|\nabla f|^2\,d\mu
 \end{equation*} 
 for every (weakly) differentiable function $f:E \to \mathbb{R}$ and where $\mathrm{Var}(f)$ is the variance with respect to the probability measure $\mu$.
\begin{proof}[Proof of Theorem \ref{thm:concentration}]
	The proof of the first concentration bound also uses Lemma \ref{lem: fundamental lemma}.
	For simplicity, let $C_{F,T} = \exp(Te^{L_FT})$ denote the constant factor appearing in the right hand side of~\eqref{eq:fundamental inequality}.
	It follows by \eqref{eq:fundamental inequality} that
	\begin{equation*}
		P\left({\mathcal W}_p(L^n(\Y_t),\cL(Y_t))\ge\varepsilon \right)\le P\left({\mathcal W}_p(L^n(\tilde \Y_t),\cL(Y_t))\ge \varepsilon / C_{F,T} \right).
	\end{equation*}
	Fix $k>2p$ such that $E[|G|^k]<\infty$, let $q\in (2p, k)$.
	By Lemma \ref{lem: moment d+5}, we have $E[|Y_t|^q]<\infty$ and thus we can apply \cite[Theorem 2]{Fou-Gui15}, to obtain the inequality
	\begin{equation*}
		P\left({\mathcal W}_p(L^n(\tilde \Y_t),\cL(Y_t))\ge \varepsilon_{F,T} \right) \le C(a_{n,\varepsilon_{F,T}}1_{\{\varepsilon_{F,T}\le 1\}} + b_{n,\varepsilon_{F,T}}),
	\end{equation*}
	where $\varepsilon_{F,T} = \varepsilon / C_{F,T}$.
	This concludes the proof of \eqref{eq:conent bound}.

	As for the proof of \eqref{eq:concent 2 proc}, 
	consider the representation of the system \eqref{eq:bsde system} given in \eqref{eq:bsde system_vector} on the probability space $(\Omega^n, \mathcal{F}^n,P)$. 
	Since $G^i$ is Lipschitz continuous for each $i$, it is easily checked that the random variable $\bG$ is again $L_G$-Lipschitz continuous.
	In fact, given $\boldsymbol\omega,\boldsymbol\theta\in \Omega^n$, we have
	\begin{align*}
	 	|\bG(\boldsymbol\omega) - \bG(\boldsymbol\theta)|^2 
		= \sum_{i=1}^n|G^i(\omega^i)-G^i(\theta^i)|^2  \le L_G\sum_{i =1}^n||\omega^i - \theta^i||^2_\infty 
		= L_G ||\boldsymbol\omega - \boldsymbol\theta||^2_\infty,
	 \end{align*} 
	where for every $\boldsymbol\omega \in \Omega^n$, $||\boldsymbol\omega||_\infty = \sup_{t \in [0,T]}\left(\sum_{i =1}^n|\omega^i(t)|^2\right)^{1/2}$. 
	Similarly, one shows that the function $\mathbf{F}$ is $L_F$-Lipschitz continuous.
	In particular, the Lipschitz constants of $\mathbf{F}$ and $\mathbf{G}$ do not depend on $n$.
	By \cite[Theorem 1.2]{T2bsde}, the law of $\Y_t$ satisfies Talagrand's $T_2$-inequality with the constant $$C_{F,G,T}:=2(L_G+TL_F)^2e^{2TL_F}.$$
	Thus, it follows by \cite[Theorem 1.3]{Gozlan09} that there is a constant $C>0$ such that for every $1$-Lipschitz continuous functions $f: C([0,T], \mathbb{R}^m)^n \to \mathbb{R}$ we have
	\begin{equation*}
	 	P\left(f(\Y)- E[f(\Y)]  \ge \varepsilon \right) 	\le e^{-\varepsilon^2C}.
	\end{equation*}
	The function $\boldsymbol{\omega}:=(\omega^1,\dots \omega^n)\mapsto\sqrt{n} {\mathcal W}_{2, ||\cdot||_\infty}(L^n(\boldsymbol\omega), \cL(Y))$ is $1$-Lipschitz continuous on $\Omega^n$.
	Thus, we have 
	\begin{align}
	 	P\left({\mathcal W}_{2,\|\cdot\|_\infty}(L^n(\Y), \cL(Y)) - E[{\mathcal W}_{2,\|\cdot\|_\infty}(L^n(\Y), \cL(Y))]\ge\varepsilon \right)&\le e^{-C\varepsilon^2n},
	 	\label{eq:conce 1st}
	\end{align}
	from which we deduce \eqref{eq:concent 2 proc}.

	Lastly, we turn to the proof of \eqref{eq:concen path}.
	If $F$ does not depend on $z$, we do not need the change of measure to get \eqref{eq:alpha beta estimate}.
	In fact, a direct estimation yields
	\begin{align*}
		\tilde Y^i_t - Y_t^{i,n} = E\left[\int_t^T \left(F_u(\tilde Y^i_u, \cL(Y_u)) - F_u(Y^{i,n}_u, L^n(\Y_u)) \right) \,du \mid {\mathcal{F}}_t^n \right].
	\end{align*}
	By Lipschitz continuity of $F$ and Gronwall's inequality we have
	\begin{align}
		\label{eq:estim y to get sup}
		|\tilde Y^i_t - Y_t^{i,n}|^2 \le e^{2TL_F}E\Big[\int_t^T {\mathcal W}_2(L^n(\Y_u),\cL(Y_u)) \,du \mid {\mathcal{F}}_t^n \Big]^2.
	\end{align}
	Thus, it follows by triangle inequality and definition of Wasserstein distance with respect to the supremum norm that
	\begin{align*}
		{\mathcal W}^2_{2,||\cdot||_\infty}(L^n(\Y), \cL(Y)) &\le {\mathcal W}^2_{2,||\cdot||_\infty}(L^n(\Y), L^n(\tilde \Y)) + {\mathcal W}^2_{2,||\cdot||_\infty}(L^n(\tilde \Y), \cL(Y)) \\
				&\le \sup_{t\in [0,T]}\frac 1n \sum_{i=1}^n|\tilde Y^i_t - Y_t^{i,n}|^2 +  {\mathcal W}^2_{2,||\cdot||_\infty}(L^n(\tilde \Y), \cL(Y))\\
				&\le \sup_{t\in [0,T]}e^{2TL_F}E\Big[\int_0^T {\mathcal W}_2(L^n(\Y_u),\cL(Y_u))\,du \mid {\mathcal{F}}_t^n \Big]^2+ {\mathcal W}^2_{2,||\cdot||_\infty}(L^n(\tilde \Y), \cL(Y))
	\end{align*}
	where the last inequality follows from \eqref{eq:estim y to get sup}.
	Hence, it follows by Doob's maximal inequality that
	\begin{align*}
		E\Big[{\mathcal W}^2_{2,||\cdot||_\infty}(L^n(\Y), \cL(Y)) \Big] &\le e^{2TL_F}TE\Big[\int_0^T {\mathcal W}_2^2(L^n(\Y_u),\cL(Y_u))\,du \Big] + E\Big[ {\mathcal W}^2_{2,||\cdot||_\infty}(L^n(\tilde \Y), \cL(Y)) \Big]\\
		&\le Cr_{n,m,k,2} + E\Big[ {\mathcal W}^2_{2,||\cdot||_\infty}(L^n(\tilde \Y), \cL(Y)) \Big],
	\end{align*}
	for some $q \in (p,k)$
	where the last inequality follows by Fubini's theorem and Theorem \ref{thm:mom bound}.
	Since $\tilde \Y = (\tilde Y^1, \dots, \tilde Y^n)$, are i.i.d., it holds $E\Big[ {\mathcal W}^2_{2,||\cdot||_\infty}(L^n(\tilde \Y), \cL(Y)) \Big] \to 0$ as $n$ goes to infinity. 
	Thus, there is and integer $\tilde n_0$ large enough such that for all $n\ge \tilde n_0$ we have $E\Big[ {\mathcal W}^2_{2,||\cdot||_\infty}(L^n(\tilde \Y), \cL(Y)) \Big]\le Cr_{n,m,k,2}$.
	Hence, we have
	\begin{equation}
	\label{eq:almost, 2nd}
		E\Big[{\mathcal W}^2_{2,||\cdot||_\infty}(L^n(\Y), \cL(Y)) \Big] \le Cr_{n,m,k,2}
	\end{equation}
	for $n\ge \tilde n_0$.
	Now, by \eqref{eq:conce 1st}, it holds that
	\begin{align*}
		&P\left({\mathcal W}_{2,||\cdot||_\infty}(L^n(\Y), \cL(Y) )\ge \varepsilon \right)\\
		&\le P\left({\mathcal W}_{2,\|\cdot\|_\infty}(L^n(\Y), \cL(Y)) - E[{\mathcal W}_{2,\|\cdot\|_\infty}(L^n(\Y), \cL(Y))]\ge\varepsilon /2\right)
		 + P\left(E\Big[{\mathcal W}_{2,||\cdot||_\infty}(L^n(\tilde \Y), \cL(Y))\Big] \ge \varepsilon/2 \right)\\
		 &\le e^{-C\varepsilon^2n} + P\left(E\Big[{\mathcal W}_{2,||\cdot||_\infty}(L^n(\tilde \Y), \cL(Y))\Big] \ge \varepsilon/2 \right).
	\end{align*}
	In view of \eqref{eq:almost, 2nd} and the fact that $r_{n,m,k,p}\downarrow 0$ as $n$ goes to infinity, we can choose $n_0 \ge \tilde n_0$ large enough such that for all $n \ge n_0$
	\begin{equation*}
		P\left(E\Big[{\mathcal W}_{2,||\cdot||_\infty}(L^n(\tilde \Y), \cL(Y))\Big] \ge \varepsilon/2 \right) = 0.
	\end{equation*}
	This concludes the proof of \eqref{eq:concen path}.
\end{proof}

\subsection{Proofs for Subsection \ref{sec:inter-bsde}}
We begin with the proof of Theorem \ref{thm:process conv}.
As we will see below, this result is obtained as a consequence of Theorem \ref{thm:mom bound}.
\begin{proof}[Proof of Theorem \ref{thm:process conv}]
Since $Y^{1,n}$ and $Y^1$ satisfy \eqref{eq:bsde system} and \eqref{eq:MkV BSDE} respectively, we have 
\begin{align*}
	Y^{1,n}_t - Y^1_t = E\left[\int_t^T \left( F_u(Y^{1,n}_u, Z^{1,n}_u, L^n(\Y_u)) -  F_u(Y^1_u,Z^1_u,\cL(Y_u)) \right) \,du \mid {\mathcal{F}}_t^n\right]
\end{align*}
so that by Lipschitz continuity of $F$, and Gronwall's inequality it holds
\begin{align*}
	|Y^{1,n}_t - Y^1_t| \le e^{L_FT}E\left[\int_t^T L_F  \left({\mathcal W}_2(L^n(\Y_u), \cL(Y_u)) + | Z^{1,n}_u - Z^1_u| \right)\,du \mid {\mathcal{F}}^n_t\right].
\end{align*}
Therefore, it follows by Doob's maximal inequality that
\begin{align}
\nonumber
	E\left[\sup_{t \in [0,T]}|Y^{1,n}_t - Y^1_t|^2 \right] 
	&\le L_F^2 e^{2L_FT}E\left[\sup_{t \in [0,T]}E\left[\int_0^T\left({\mathcal W}_2(L^n(\Y_u), \cL(Y_u)) + | Z^{1,n}_u - Z^1_u| \right)\,du \mid {\mathcal{F}}_t^n\right]^2 \right]
	\\
	\notag
	&\le L_F^2 e^{2L_FT}E\left[\left(\int_0^T\left({\mathcal W}_2(L^n(\Y_u), \cL(Y_u)) + | Z^{1,n}_u - Z^1_u| \right)\,du \right)^2\right]
	\\
	\label{eq:estim sup Y}
	&\le 2 L_F^2 T  e^{2L_FT} \left(\int_0^TE\left[{\mathcal W}_2^2(L^n(\Y_u), \cL(Y_u))\right]\,du + E\left[\int_0^T| Z^{1,n}_u - Z^1_u|^2\,du \right] \right).
\end{align}
On the other hand, applying It\^{o}'s formula to the process $| Y^{1,n}_t-Y^1_t|^2$, we have
	\begin{align*}
		| Y^{1,n}_t-  Y_t^1|^2& = -2\sum_{j=1}^n\int_t^T(Y^{1,n}_s-Y^1_s)( Z^{1,j,n}_s-\delta_{1j}Z^1_s)dW^j_s
		\\
		&\quad -\sum_{j,l=1}^n\int_t^T (Z^{1,j,n}_s-\delta_{1j}Z^1_s)(Z^{1,l,n}_s- \delta_{1l}Z^1_s)d\langle W^j, W^l\rangle_s
		\\
		&\quad+2\int_t^T( Y^{1,n}_s-Y^1_s)\left\{F_s(Y^{1,n}_s,Z^{1,1,n}_s, L^n(\Y_s)) - F_s( Y^1_s, Z^1_s,\cL(Y^1_s))\right\}ds.
	\end{align*}
	By Lipschitz continuity of $F$ we then have
	\begin{align}
	\notag
		&| Y^{1,n}_t-Y_t^1|^2+\sum_{j,l=1}^n\int_t^T (Z^{1,j,n}_s-\delta_{1j}Z^1_s)(Z^{1,l,n}_s-\delta_{1l}Z^1_s)d\langle W^j, W^l\rangle_s\\\notag
		&\quad
		\leq  -2\sum_{j=1}^n\int_t^T( Y^{1,n}_s-Y^1_s)(Z^{1,j,n}_s-\delta_{1j}Z^1_s)dW^j_s\\\notag
		&\quad +2\int_t^TL_F\abs{ Y^{1,n}_s-Y^1_s}\abs{ Z^{1,1,n}_s-Z^1_s}ds 
		+2\int_t^TL_F\abs{Y^{1,n}_s-Y^{1}_s}^2ds\\\notag
		&\quad +2\int_t^TL_F\abs{ Y^{1,n}_s - Y^1_s}{\mathcal W}_2(L^n(\Y_s),\cL(Y^1_s))\,ds\\\notag
		&\le  -2\sum_{j=1}^n\int_t^T( Y^{1,n}_s-Y^1_s)(Z^{1,j,n}_s-\delta_{1j}Z^1_s)dW^j_s+\int_t^T\frac{L_F}{\alpha}\abs{ Z^{1,1,n}_s-Z^1_s}^2ds\\
		\label{eq:estim y&z}
		&\quad 
		+\int_t^T (3\alpha + L_F) \abs{ Y^{1,n}_s-Y^{1}_s}^2ds+\int_t^TL_F{\mathcal W}^2_2(L^n(\Y_s),\cL(Y^1_s))\,ds,
	\end{align}
	where the last inequality follows by Young's inequality with some constant $\alpha>0$.
	Choosing $\alpha=L_F+1$, we obtain
	\begin{align*}
		|Y^{1,n}_t - Y^1_t|^2 \le E\left[ \int_t^T (4 + L_F) L_F | Y^{1,n}_s-Y^{1}_s|^2ds+\int_t^TL_F{\mathcal W}^2_2(L^n(\Y_s),\cL(Y^1_s))\,ds \mid {\mathcal{F}}_t^n \right] 
	\end{align*}
	so that by Gronwall's inequality we have that
	\begin{equation}
	\label{eq:first estim Y1}
		|Y^{1,n}_t - Y^1_t|^2 \le e^{ (4 + L_F) L_F T}E\left[  \int_t^TL_F{\mathcal W}^2_2(L^n(\Y_u),\cL(Y_u^1))\,du \mid {\mathcal{F}}_t^n \right] .
	\end{equation}
	Now, integrating on both sides yields
	\begin{align}
	\nonumber
		E\left[ |Y^{1,n}_t - Y^1_t|^2 \right] &\le e^{ (4 + L_F) L_F T}E\left[  \int_0^TL_F{\mathcal W}^2_2(L^n(\Y_u),\cL(Y^1_u))\,du  \right]\\
		\label{eq:estim y by w22}
		&\le e^{ (4 + L_F) L_F T}L_F \int_0^TE\left[{\mathcal W}^2_2(L^n(\Y_u),\cL(Y^1_u))  \right]\,du
	\end{align}
	from which we derive, due to Theorem \ref{thm:mom bound}, that
	\begin{equation}
	\label{eq:diff for y}
		E\left[ |Y^{1,n}_t - Y^1_t|^2 \right]\le e^{ (4 + L_F) L_F T}L_FTC \, r_{n,m,q,2}
	\end{equation}
	with $q \in (2,k)$.

	To get the convergence estimate for the control processes $Z^{1,n}$, notice that by \eqref{eq:estim y&z} (with the choice $\alpha=L_F+1$), we have
	\begin{align}
	\nonumber
		E\left[ \int_0^T|Z^{1,n}_u - Z^1_u|^2\,du \right] &\le (L_F+1) E\left[  (4 + L_F) L_F \int_0^T|Y^{1,n}_u - Y^1_u|^2\,du + L_F\int_0^T{\mathcal W}^2_2(L^n(\Y_u),\cL(Y^1_u))\,du \right]
		\\\notag
		&\le (L_F+1) \Big(  (4 + L_F) L_F \int_0^TE\left[|Y^{1,n}_u - Y_u^1|^2\right]\,du 
		\\
		&\quad+ L_F\int_0^TE\left[{\mathcal W}^2_2(L^n(\Y_u),\cL(Y_u^1)) \right]\,du \Big).
		\label{eq: estim z by W22}
	\end{align}
	Now combine this with the inequalities \eqref{eq:estim sup Y}, \eqref{eq:diff for y} and 
	 Theorem \ref{thm:mom bound} to conclude.
\end{proof}
We now turn to the particular case of systems with linear interaction.
Unlike Theorem \ref{thm:process conv}, the proof of Proposition \ref{prop:lin inter} does not follow from the convergence of the empirical measure of $\Y$, but by a direct argument which seems tailor-made for ``linear interaction functions''.
Before presenting the proof, let us justify the well-posedness of the system.
\begin{proof}[Proof of Remark \ref{rem:lin syst exists}]
	The existence and uniqueness of square integrable solutions follows, as in the non-linear interaction case, from \cite{karoui01} for the system~\eqref{eq:gen lin inter} and \cite[Theorem 4.23]{MR3752669} for the McKean-Vlasov BSDE~\eqref{eq:lin inter McKVl}. 
	It suffices to show that the generators are Lipschitz continuous (with Lipschitz constant independent of $n$).
	We give only the argument for the McKean-Vlasov equation. 
	Let $y,y' \in \mathbb{R}^m$, $z, z' \in \mathbb{R}^{m \times d}$ and $\mu, \mu'\in \cP(\mathbb{R}^m)$. By (Lip) we have
	\begin{align*}
		|F(y,z,\mu) - F(y',z' ,\mu') | &\le L_F\Big(|y-y'| + |z-z'| + |\int_{\mathbb{R}^m}f(y,a,z)\,\mu(da) - \int_{\mathbb{R}^m}f(y',a,z')\,\mu'(da) | \Big)
		\\
					&\le L_F\Big(|y-y'| + |z-z'| + L_f \left |\int_{\mathbb{R}^m}\frac{1}{L_f}f(y,a,z)\,\mu(da) - \int_{\mathbb{R}^m}\frac{1}{L_f}f(y,a,z)\,\mu'(da) \right|
					\\
					&\quad  + \left| \int_{\mathbb{R}^m}f(y,a,z)\,\mu'(da) - \int_{\mathbb{R}^m}f(y',a,z')\,\mu'(da) \right| \Big)
					\\
					&\le L_F\Big(|y-y'| + |z- z'| + L_f\cW_1(\mu, \mu') + L_f\int_{\mathbb{R}^m} (|y-y'| + |z-z'|) \,\mu(da) \Big)
					\\
					&\le \max(L_F, L_F L_f, L_f) \Big(|y-y'| + |z- z'| + \cW_2(\mu, \mu')  \Big).
	\end{align*}
	 To derive the penultimate inequality, we used Kantorovich-Rubinstein duality formula.
\end{proof}

\begin{proof}[Proof of Proposition \ref{prop:lin inter}]
It follows by application of It\^o's formula that
\begin{align}
	\notag
		&| Y^{1,n}_t-Y_t^1|^2+\sum_{j,l=1}^n\int_t^T (Z^{1,j,n}_s- \delta_{1j}Z^1_s)(Z^{1,l,n}_s-\delta_{1l}Z^1_s)d\langle W^j, W^l\rangle_u
		\leq 
		\\
		\notag
		& -2\sum_{j=1}^n\int_t^T( Y^{1,n}_u-Y^1_u)(Z^{1,j,n}_u- \delta_{1,j} Z^1_u)dW^j_u
		+2\int_t^T(Y^{1,n}_u-Y^1_u)\left( F_u\Big(Y^{1,n}_u, Z^{1,n}_u,\frac1n \sum_{j=1}^nf_u(Y^{1,n}_u,  Y^{j,n}_u,Z^{1,n}_u)\Big)\right.
		\\\notag
		&\quad  -  F_u\Big(Y^1_u, Z^1_u,\left.\int_{\mathbb{R}^m} f_u(Y^1_u, y,Z^1_u) \, d\cL(Y^1_u)(y)\Big)\right)du
		\\\notag
		 &\le-2\sum_{j=1}^n\int_t^T( Y^{1,n}_u-Y^1_u)(Z^{1,j,n}_u-\delta_{1j}Z^1_u)dW^j_u + L_F\int_t^T \frac{1}{\alpha}|Z^{1,n}_u - Z^1_u |^2\,du
		+L_F\int_t^T (2\alpha+2)  |Y^{1,n}_u-Y^1_u|^2\,du
		\\
		&\quad + L_F \int_t^T\frac{1}{\alpha}\left|\frac1n \sum_{j=1}^nf_u(Y^{1,n}_u,  Y^{j,n}_u, Z^{1,n})  - \int_{\mathbb{R}^m} f_u(Y^1_u, y,Z^{1}_u)\, d\cL(Y^1_u)(y)\right|^2du
		\label{eq:diff y and z lin}
	\end{align}
	where the last inequality follows by Young's inequality with some constant $\alpha>0$.
	Let us analyze the last term above.
	It follows by triangle inequality that
	\begin{align*}
		&\left|\frac1n \sum_{j=1}^nf_u(Y^{1,n}_u,  Y^{j,n}_u, Z^{1,n})  - \int_{\mathbb{R}^m} f_u(Y^1_u, y,Z^{1}_u)\, d\cL(Y^1_u)(y)\right|^2\\
		&\le 2\left|\frac1n \sum_{j=1}^n\left\{f_u(Y^{1,n}_u, Y^{j,n}_u,Z^{1,n}_u) - f_u(Y^{1}_u, Y^{j,n}_u, Z^{1,n}_u)\right\}\right|^2\\
		&\quad +4\left|\frac1n \sum_{j=1}^n\left(f_u(Y^{1}_u, Y^{j,n}_u,Z^{1,n}_u) - f_u(Y^{1}_u, Y^{j}_u, Z^{1,n}_u)\right)\right|^2
		 + 8\left|\frac1n \sum_{j=1}^n\left(f_u(Y^{1}_u, Y^{j}_u,Z^{1,n}_u) - f_u(Y^{1}_u, Y^{j}_u, Z^{1}_u)\right)\right|^2\\
		&\quad +8 \left|\frac1n \sum_{j=1}^n\left(f_u(Y^{1}_u,  Y^{j}_u,Z^{1}_u)  - \int_{\mathbb{R}^m} f_u(Y^1_u, y,Z^1_u)\, d\cL(Y^1_u)(y)\right)\right|^2\\
		&\le L_f\left\{2 |Y^{1,n}_u - Y^1_u|^2+ 8 |Z^{1,n}_u - Z^1_u|^2\right\} + 4 L_f\frac1n\sum_{j=1}^n|Y^{j,n}_u - Y^j_u|^2\\
		&\quad
		 + \frac{8}{n^2}\left| \sum_{j=1}^n \left[f_u(Y^{1}_u,  Y^{j}_u, Z^1_u)  - \int_{\mathbb{R}^m} f_u(Y^1_u, y,Z^1_u)\, d\cL(Y^1_u)(y) \right] \right|^2.
	\end{align*}
	Coming back to \eqref{eq:diff y and z lin}, this allows to obtain, after taking conditional expectation with respect to the sigma algebra ${\mathcal{F}}_t^n$,
	\begin{align*}
		&| Y^{1,n}_t-Y_t^1|^2+E\left[\int_t^T|Z^{1,n}_s-Z^1_s|^2ds\mid{\mathcal{F}}_t^n\right]\\
		&\quad \le E\left[\int_t^T \left\{ (L_F(2\alpha +2)+ \frac{2L_f}{\alpha})|Y^{1,n}_u - Y^1_u|^2+ \frac{8L_f+L_F}{\alpha}|Z^{1,n}_u - Z^1_u|^2 \right\}\,du \mid{\mathcal{F}}_t^n\right]\\
		&\qquad + \frac{4L_f}{\alpha}E\left[\frac1n\sum_{j=1}^n\int_t^T|Y^{j,n}_u - Y^j_u|^2\,du\mid {\mathcal{F}}_t^n\right]\\
		&\qquad+ \frac{8}{\alpha n^2}E\left[\int_t^T\left| \sum_{j=1}^n \left[ f_u(Y^{1}_u,  Y^{j}_u,Z^1_u)  - \int_{\mathbb{R}^m} f_u(Y^1_u, y,Z^1_u)\, d\cL(Y^1_u)(y)\right|^2 \right] \,du\mid{\mathcal{F}}_t^n\right].
	\end{align*}
	Choose $\alpha = 8L_f+L_F+1$.
	Then, we have
	\begin{align}
	\notag
		&\alpha| Y^{1,n}_t-Y_t^1|^2+E\left[\int_t^T|Z^{1,n}_s-Z^1_s|^2ds\mid{\mathcal{F}}_t^n\right] 
		\\\notag
		&\le \alpha E\left[\int_t^T\left\{ (L_F (2\alpha+2)  +1)|Y^{1,n}_u - Y^1_u|^2 \right\}\,du \mid{\mathcal{F}}_t^n\right]
		\\\notag
		& \quad+ 4 L_fE\left[\frac1n\sum_{j=1}^n\int_t^T|Y^{j,n}_u - Y^j_u|^2\,du\mid {\mathcal{F}}_t^n\right]
		\\
		&\quad + \frac{8}{ n^2}E\left[\int_t^T\left| \sum_{j=1}^n \left[ f_u(Y^{1}_u,  Y^{j}_u,Z^1_u)  - \int_{\mathbb{R}^m} f_u(Y^1_u, y,Z^1_u)\, d\cL(Y^1_u)(y)\right] \right|^2\,du\mid{\mathcal{F}}_t^n\right]
		 \label{eq:diff y and z lin second}.
	\end{align}

	Since $E[|Y^{1,n}_t - Y^1_t|^2] = E[|Y^{j,n}_t-Y^j_t|^2]$ for all $j$
	we can apply Gronwall's inequality to get 
	\begin{align*}
		|Y^{1,n}_t - Y_t^1|^2 &\le \frac{C_{T,F,f}}{n^2}E\left[\int_0^T\left| \sum_{j=1}^n \left[ f_u(Y^{1}_u,  Y^{j}_u,Z^1_u)  - \int_{\mathbb{R}^m} f_u(Y^1_u, y,Z^1_u)\, d\cL(Y^1_u)(y) \right] \right|^2\,du\mid{\mathcal{F}}_t^n\right]
	\end{align*}
	for a constant $C_{T,F,f}$ depending only on $T,L_f$ and $L_F$.
	Now, denoting by $f^\ell$ the $\ell^{th}$ component of $f$ and using 
	 that $Y^i_u$ and $Y^j_u$ have the same law we deduce that 
	\begin{align*}
		E\left[|Y^{1,n}_t - Y_t^1|^2\right]&\le \frac{C_{T,F,f}}{n^2}E\left[\int_0^T\left| \sum_{j=1}^n \left[ f_u(Y^{1}_u,  Y^{j}_u,Z^1_u)  - \int_{\mathbb{R}^m} f_u(Y^1_u, y,Z^1_u)\, d\cL(Y^1_u)(y) \right] \right|^2\,du\right]
		\\
		 &\le \frac{C_{T,F,f}}{n^2}\int_0^TE\Bigg[\sum_{\ell=1}^m \sum_{j,k=1}^n\Big(f^\ell_u(Y^{1}_u,  Y^{j}_u,Z^1_u)  - \int_{\mathbb{R}^m} f^\ell_u(Y^1_u, y,Z^1_u)\, d\cL(Y^1_u)(y)\Big)
		 \\
		 &\quad \times\Big(f^\ell_u(Y^{1}_u,  Y^{k}_u,Z^1_u)  - \int_{\mathbb{R}^m} f^\ell_u(Y^1_u, y,Z^1_u)\, d\cL(Y^1_u)(y)\Big)\Bigg]\,du
		 \\
		&\le \frac{C_{T,F,f}}{n^2}\int_0^T\sum_{j=1}^nE\left[\left|f_u(Y^{1}_u,  Y^{j}_u,Z^1_u)  - \int_{\mathbb{R}^m} f_u(Y^1_u, y,Z^1_u)\, d\cL(Y^1_u)(y)\right|^2\right]
		\\
		&= \frac{C_{T,F,f}}{n^2}\int_0^T\sum_{j=1}^nE\left[\text{Var}(f_u(Y^1_u,Y_u^j,Z^1_u)) \right]\,du.
	\end{align*}
	The equality before the last one follows from the fact that $Y^i_t$, $j=1,\dots, n$ are i.i.e. for all $t$. 
	Since the law $\cL(Y^1_u)$ of $Y^1_u$ satisfies Talagrand's $T_2$ inequality with a constant $C_{F,f,G,T}$ which depends only on the Lipschitz constants of $F,f$ and $G$ and of $T$ (and which does not depend on the dimensions) (see \cite[Theorem 1.3]{T2bsde}) it follows e.g. by \cite[Theorem 22.17]{Vil2} (see also \cite[Section 7]{Otto-Villani}) that $\cL(Y^1_u)$ satisfies the Poincar\'e inequality with the same constant.
	That is, it holds that
	\begin{equation*}
		\text{Var}(f_u(x,Y_u^j,z)) \le C_{F,f,G,T}\int_{\mathbb{R}^m}|\partial_yf_u(x,y,z)|^2\,d\cL(Y^j_u)(y)\quad \text{for all } x,z \text{ fixed}.
	\end{equation*}
	Since $f$ is Lipschitz continuous, $L_f^2$ is an upper bound for the integral in the right hand side above (uniformly in $x,z$).
	Therefore, we have 
	\begin{align*}
		E\left[|Y^{1,n}_t - Y_t^1|^2\right] &\le \sum_{j=1}^n\frac{C}{n^2}\int_0^TE\left[\int_{\mathbb{R}^m}|\partial_yf_u(Y^j_u,y,Z^1_u)|^2\, d\cL(Y^j_u)(y)\right]\,du\\
			&\le \frac{C}{n}
	\end{align*}
for some constant $C$ depending on $T, L_F$ $L_f$ and the Lipschitz constant of $G$.

Now, showing that $E[\int_0^T|Z^{1,n}_u - Z^1_u|^2\,du]\le C/n$ follows by \eqref{eq:diff y and z lin second}.
In fact, that inequality implies
\begin{align*}
	&E\left[\int_0^T|Z^{1,n}_s-Z^1_s|^2ds\right]
		\leq  \int_0^TCE\left[|Y^{1,n}_s-Y^1_s|^2\right]\,ds\\
		&\quad + \frac{8}{n^2}E\left[\int_0^T\left|\frac1n \sum_{j=1}^nf_u(Y^{1,n}_u,  Y^{j,n}_u, Z^{1,n})  - \int_{\mathbb{R}^m} f_u(Y^1_u, y,Z^{1}_u)\,d\cL(Y^1_u)(y)\right|^2ds\right].
	\end{align*}
	We have just seen that, up to the factor $CT$ (for some constant $C>0$),
	the first term is smaller than the second one, which in turn is less than $C/n$ for some constant $C>0$.
	This completes the proof.
\end{proof}

\subsection{Proof of Theorem \ref{thm:lim PDE}}

\begin{proof}[Proof of Theorem \ref{thm:lim PDE}]
	First note that the PDE \eqref{eq:PDE n system} admits a (classical) solution.
	Assumption (PDE1) says that $B$ is differentiable on $\cP_2(\RR^d)$.
	Thus, it follows by \cite[Proposition 5.35]{MR3752669} that, for every $x \in \RR^d$, the projection of $B$ on $\RR^d\times(\RR^d)^n$ given by $ (x, \x) \mapsto \bB(x,\x) := B(x,L^n(\x))$ is again differentiable, with
	\begin{equation}
	\label{eq:War diff to Euc diff}
		\partial_{x_i}\bB(\cdot, \x)
		 = \frac1n \partial_\mu B(\cdot, L^n(\x))(x_i), \quad i=1,\dots, n.
	\end{equation}

	One similarly shows that the respective projections $\bsigma$, $\bF$ and $\bG$ of $\sigma$, $F$ and $G$ on finite dimensional spaces (with appropriate dimensions) are three times differentiable, and by the identity \eqref{eq:War diff to Euc diff} the derivatives of first and second order are bounded by $L_F$.
	In particular, the bound does not depend on $n$.
	Therefore, it follows from \cite[Theorem 3.2]{PardouxPeng92} that the PDE \eqref{eq:PDE n system} admits a solution $v^{i,n}:[0,T]\times (\RR^d)^n\to \RR^m$ with the probabilistic representation
		\begin{equation*}
		v^{i,n}(s,X^{1,n,t,x_1}_s, \dots,X^{n,n,t,x_n}_s) = Y^{i,n,t,\x}_s\quad \text{for all}\quad s\ge t,\,\, (t,\x)\in [0,T]\times (\RR^d)^n,
	\end{equation*}
	where $(X^{i,n,t,x_i},Y^{i,n,t,\x},Z^{i,n,t,\x})_{i=1,\dots,n, t \ge 0}$ solves the decoupled FBSDE
	\begin{equation}
	\label{eq:n particles fbsde}
		\begin{cases}
			X^{i,n,t,x_i}_s = x_i + \int_t^s B(X^{i,n,t,x_i}_u, L^n(\X^{t, \x}_u))\,du + \int_t^s\sigma(X^{i,n,t,x_i}_u,L^n(\X^{t, \x}_u))\,dW_u\\
			Y^{i,n,t,\x}_s = G(X^{i,n,t,x_i}_T,L^n(\X^{t, \x}_T)) + \int_s^T F(X^{i,n,t,x_i}_u, Y^{i,n,t,\x}_u, Z^{i,n,t,\x}_u,L^n(\X^{t,\x}_u), L^n(\Y^{t,\x}_u) )\,du\\
			\quad - \int_s^TZ^{i,n,t,\x}_u\,dW_u\\
			i= 1,\dots,n
		\end{cases}
	\end{equation}
	with $\X^{t,\x}:=(X^{1,n,t,x_1}_u, \dots, X^{n,n,t,x_n}_u)$ and $\Y^{t,\x}: = (Y^{1,n,t,\x}_u,\dots, Y^{n,n,t,\x}_u)$.
	In particular $v^{i,n}(t,\x) = Y^{i,n,t,\x}_t$.
	Naturally, it follows by standard SDE and BSDE theories (see e.g. \cite{PP90}) that the system \eqref{eq:n particles fbsde} is well-posed, since by (PDE1) and the identity \eqref{eq:War diff to Euc diff}, the functions $\bB,\bsigma$ and $\bF$ are Lipschitz continuous (recall they are defined on finite dimensional spaces).

	On the other hand, the PDE \eqref{eq:master pde} is connected to the following decoupled McKean-Vlasov FBSDE:
	\begin{equation}
	\label{eq:master fbsde}
		\begin{cases}
			X^{t,\xi}_s = \xi + \int_t^s B(X^{t,\xi}_u,\cL(X_u^{t,\xi}))\,du + \int_t^s\sigma(X^{t,\xi}_u,\cL(X^{t, \xi}_u))\,dW^1_u\\
			Y^{t,\xi}_s = G(X^{t,\xi}_T, \cL(X^{t,\xi}_T)) + \int_s^T F(X^{t,\xi}_u, Y^{t,\xi}_u, Z^{t,\xi}_u, \cL(X^{t,\xi}_u), \cL(Y^{t,\xi}_u))\,du - \int_s^TZ^{t,\xi}_u\,dW_u,
		\end{cases}
	\end{equation}
	whose solution is the triple $(X^{t,\xi}, Y^{t,\xi}, Z^{t,\xi})_t$, with $\xi \in L^2(\Omega,{\mathcal{F}}_t, P; \RR^d)$ fixed.
	Since $B$ is differentiable on $\cP_2(\RR^d)$, it follows by definition of $\partial_\mu B$ that for every $\mu, \mu' \in \cP_2(\RR^d)$, one has
	\begin{equation*}
		B(x, \mu) - B(x, \mu') = \int_0^1\int_{\RR^d}\partial_\mu B(x, \lambda\mu + (1-\lambda)\mu')(y)(\mu - \mu')(dy)d\lambda.
	\end{equation*}
	Since $y\mapsto \partial_\mu B(x, \mu)(y)$ is $L_F$-Lipschitz continuous, it follows by Kantorovich-Rubinstein formula that 
	\begin{equation*}
		|B(x, \mu) - B(x,\mu')| \le \int_0^1L_F\cW_1(\mu, \mu') \,d\lambda = L_F\cW_2(\mu, \mu').
	\end{equation*}
	That is, $B$ is Lipschitz continuous. 
	One similarly shows that the functions $\sigma,F$ and $G$ are Lipschitz continuous on their respective domain.
	Thus, the equation \eqref{eq:master fbsde} is well-posed, see e.g. \cite{MR1108185} and \cite{MR3752669}.

	By \cite[Proposition 5.2]{ChassagneuxCrisanDelarue_Master}, it holds that
\begin{equation*}
	V(s,X^{t,\xi}_s, \cL(X_s^{t,\xi})) = Y^{t,\xi}_s
\end{equation*}
for all $0\le t\le s\le T$ and $\xi \in L^2(\Omega,{\mathcal{F}}_t, P; \RR^d)$.
In particular, $V(t, \xi, \mu) =  Y^{t,\xi}_t$ for all $(t,\xi,\mu)\in [0,T]\times L^2(\Omega,{\mathcal{F}}_t, P; \RR^d)\times {\mathcal{P}}_2(\RR^d)$ with $\cL(\xi) = \mu$.

Let $\xi_i$, $i = 1, \dots, n,$ be i.i.d. random variables in $L^2(\Omega, {\mathcal{F}}_t, P; \RR^d)$ with common law $\mu$ and denote $\boldsymbol{\xi}:=(\xi_1,\dots, \xi_n)$. 
Then we have 
\begin{align}
\label{eq:estim V by Y}
	E\Big[ \sup_{t \in [0,T]}| v^{1,n}(t,\xi_1, \dots,\xi_n) - V(t,\xi_1, \cL(\xi^1)) |^2 \Big]
	 = E\Big[\sup_{t \in[0,T]} | Y^{1,n,t,\boldsymbol{\xi}}_t - Y^{t,\xi_1}_t |^2 \Big].
\end{align}
To conclude, it remains to estimate the convergence rate of the right hand side of \eqref{eq:estim V by Y}.
This can be done as in the proof of Theorem \ref{thm:process conv}.
The only difference here being that the generator and terminal condition of the $n$-particle system ``depends on $(i,n)$'' through the processes $X^{i,n,t,\xi_i}$ and their empirical measure.
But thanks to the Lipschitz continuity property of $F$ and $G$ (proved above) this does not cause much problems.
One can follow the estimations in the proofs of Lemma \ref{lem: fundamental lemma} and Theorem \ref{thm:process conv}, adding the terms
\begin{equation*}
		\delta_s^n:= |X^{1,n,t,\xi_1}_s - X^{t,\xi_1}_s|^2 \quad \text{and}\quad \eta^n_s := {\mathcal W}^2_2(L^n(\X_s^{t,\boldsymbol \xi_1}), \cL(X_s^{t,\xi_1})) .
\end{equation*}
 In fact, using the arguments leading to Equation \eqref{eq:estim sup Y}, Equation \eqref{eq:estim y by w22} and Equation \eqref{eq: estim z by W22} respectively, we obtain 
	\begin{align*}
		E\Big[\sup_{t\in [0,T]}|Y^{1,n,t,\boldsymbol\xi}_t -  Y^{t,\xi_1}_t|^2\Big] &\le 16L_F^2e^{2L_FT}E\Big[\int_0^T    \left( \delta^n_u + \eta^n_u + {\mathcal W}_2^2(L^n(\Y_u^{t,\boldsymbol\xi}), \cL(Y_u^{t,\xi})) \right) \,du\\
		&\quad +\int_0^T|Z^{1,n,t,\boldsymbol\xi}_u -  Z^{t,\xi}_u|^2\,du  \Big]  + 2L_G^2  E\Big[\delta^n_T + \eta^n_T\Big],
	\end{align*}
	\begin{align}
	\notag
		E\left[ |Y^{1,n,t,\boldsymbol\xi}_s -  Y^{t,\xi}_s|^2 \right] 
		&\le e^{ (4 + L_F) L_F T} \left(L_GE\Big[\delta^n_T + \eta^n_T\Big] + L_F  E\Big[\int_0^T \left( \delta^n_u +\eta^n_u+ {\mathcal W}^2_2(L^n(\Y_u^{t,\boldsymbol\xi}),\cL(Y_u^{t,\xi})) \right) \,du \Big] \right)
		\\\notag
		&\le e^{ (4 + L_F) L_F T}(TL_F\vee L_G\vee L_F)\Bigg(E\Big[\sup_{s\in[0,T]}\delta^n_{s} +\sup_{s\in [0,T]}\eta^n_s \Big]\\
		&\quad + E\Big[\int_0^T{\mathcal W}^2_2(L^n(\Y_u^{t,\boldsymbol\xi}),\cL(Y_u^{t,\xi})) \,du \Big] \Bigg)
	\end{align}
	and
	\begin{align*}
		E\Big[ \int_0^T|Z^{1,n,t,\boldsymbol\xi}_u -  Z^{t,\xi}_u|^2\,du \Big] 
		&\le  (L_F+1) \Bigg(L_GE\Big[\sup_{s \in [0,T]}\delta^n_s + \sup_{s \in [0,T]} \eta^n_s \Big] + \int_0^T  (5 + L_F )L_FE\Big[|Y^{1,n,t,\boldsymbol\xi}_u - Y^{t,\xi}_u|^2\Big]\,du\\
		&\quad + L_F E\Big[\int_0^T {\mathcal W}^2_2(L^n(\Y_u^{t,\boldsymbol\xi}), \cL(Y_u^{t,\xi})) \,du\Big]  \Bigg).
	\end{align*}
	Therefore, there is a constant $C_{L_F,T}$ depending only on $L_F$ and $T$ such that
	 \begin{equation}
	\label{eq:forward sde}
		E\Big[\sup_{t\in [0,T]}|Y^{1,n,t,\boldsymbol\xi}_t - Y^{t,\xi}_t|^2\Big] 
		\le 
		C_{L_F,T} \Big(E\Big[\sup_{s \in [0,T]}\delta^n_s + \sup_{s \in [0,T]}\eta^n_s\Big] + \int_0^T E\Big[{\mathcal W}^2_2(L^n(\Y_u^{t,\boldsymbol\xi}), \cL(Y_u^{t,\xi})) \Big]\,du \Big).
	\end{equation}
	Moreover, by the theory of (forward) propagation of chaos, see e.g.~\cite[Theorem 2.12]{MR3753660}, it holds that
	\begin{equation}
	\label{eq:classic chaos}
		E\Big[ \sup_{s \in [0,T]}\delta^n_s + \sup_{s \in [0,T]}\eta^n_s\Big] \le C \varepsilon_n.
	\end{equation}
	It remains to estimate $E\big[{\mathcal W}^2_2(L^n(\Y_u^{t,\bxi}), \cL(Y_u^{t,\xi})) \big]$.
	This is done as follows.
	We apply It\^o's formula to $|Y^{i,n,t,\bxi}_s - Y^{t, \xi}_s|^2$.
	This yields, thanks to the Lipschitz continuity of $F$ and $G$,
	\begin{align*}
		&|Y^{i,n,t,\boldsymbol\xi}_s - Y^{t, \xi}_s|^2 \le L_F\left(\delta^n_T + \eta^n_T \right) + L_F\int_s^T \Big[ \Big(\delta^n_u + \eta^n_u + |Z^{i,n,t,\boldsymbol\xi}_u - Z^{t,\xi}_u|+ |Y^{i,n,t,\boldsymbol\xi}_u - Y^{t, \xi}_u|
		\\
		&\quad + \cW_2(L^n(\Y^{t,\boldsymbol\xi}_u), \cL(Y_u^{t,\xi}))\Big)|Y^{i,n,t,\boldsymbol\xi}_u - Y^{t, \xi}_u| - \frac12 |Z^{i,n,t,\boldsymbol\xi}_u - Z^{t, \xi}_u|^2 \Big]\,du
		\\
		&\quad -\int_s^T(Y^{i,n,t,\boldsymbol\xi}_u - Y^{t, \xi}_u)(Z^{i,n,t,\boldsymbol\xi}_u - Z^{t, \xi}_u)\,dW_u.
	\end{align*}
	Thus, using sucessively Young's inequality with a constant $\alpha>0$ and Gronwall's inequality we have
	\begin{align*}
		&|Y^{i,n,t,\boldsymbol\xi}_s - Y^{t, \xi}_s|^2 \le C_{L_F, \alpha,T}E\Big[\left(\delta^n_T + \eta^n_T \right)
		+  \int_s^T\Big(\delta^n_u + \eta^n_u  + \cW_2^2(L^n(\Y^{t,\boldsymbol\xi}_u), \cL(Y_u^{t,\xi}))\Big)\\
		&\quad + \Big(\frac{1}{2\alpha}- \frac12\Big) |Z^{i,n,t,\boldsymbol\xi}_u - Z^{t, \xi}_u|^2\,du \mid {\mathcal{F}}_s \Big]\\
		&\le C_{L_F,T}E\Big[\sup_{s\in[0,T]}\delta^n_s + \sup_{s \in [0,T]}\eta^n_s 
		+  \int_0^T \cW_2^2(L^n(\Y^{t,\boldsymbol\xi}_u), \cL(Y_u^{t,\xi}))\mid {\mathcal{F}}_s \Big],
	\end{align*}
	where the second inequality follows by choosing $\alpha>1$ and for some constant $C_{L_F,T}>0$.
	Since 
	\begin{equation*}
		\cW_2^2(L^n(\Y^{t,\xi}_u), \cL(Y_u^{t,\xi})) \le \frac1n\sum_{i=1}^n|Y^{i,n,t,\boldsymbol\xi}_u - Y^{t,\xi}_u|^2,
	\end{equation*}
	we obtain by Gronwall's inequality that
	\begin{equation*}
		E\Big[\cW_2^2(L^n(\Y^{t,\boldsymbol\xi}_s), \cL(Y_s^{t,\xi}))\Big] \le C_{L_F,T}E\Big[\sup_{s\in[0,T]}\delta^n_s + \sup_{s \in [0,T]}\eta^n_s \Big]
	\end{equation*}
	and it thus follows from \eqref{eq:forward sde} and \eqref{eq:classic chaos} that
	\begin{equation*}
		E\Big[\sup_{t\in [0,T]}|Y^{1,n,t,\boldsymbol\xi}_t - Y^{t,\xi}_t|^2\Big] \le C_{L_F,T}\varepsilon_n
	\end{equation*}	
	for some constant $C_{L_F,T}>0$.
	Combining this with  \eqref{eq:estim V by Y} leads to~\eqref{eq:PDE linear inter}.

	To prove~\eqref{eq:PDE linear inter empi}, let $(\xi_1, \dots, \xi_n)$ be $n$ i.i.d. ${\mathcal{F}}_t$-measurable, square integrable random variables.
	It follows by triangle inequality that
	\begin{align*}
		 | v^{i,n}(t,\xi_1, \dots, \xi_n) - V(t, \xi_i, L^n(\boldsymbol\xi)) | &\le | v^{i,n}(t,\xi_1, \dots, \xi_n) - V(t, \xi_i, \cL(\xi_1) | + | V(t, \xi_i, \cL(\xi_1)) - V(t, \xi_i, L^n(\boldsymbol\xi)) |\\
		 &\le | v^{i,n}(t,\xi_1, \dots, \xi_n) - V(t, \xi_i, \cL(\xi_1) | + C\cW_2(\cL(\xi_1), L^n(\boldsymbol\xi)),
	\end{align*}
	where the second inequality follows by Lipschitz continuity of $V$ given in \cite[Proposition 5.2]{ChassagneuxCrisanDelarue_Master}.
	Therefore, we obtain by \eqref{eq:PDE linear inter} and \cite[Theorem 1]{Fou-Gui15} that
	\begin{equation*}
		E\left[	 | v^{i,n}(t,\xi_1, \dots, \xi_n) - V(t, \xi_i, L^n(\boldsymbol\xi)) |^2 \right] \le C_{L_F,T} \, (\varepsilon_n + r_{n,d,k,2}).
	\end{equation*}
	This concludes the proof.
\end{proof}

\bibliographystyle{abbrvnat}
\bibliography{./references-Concen_RM,./references-MFGsupp}

\vspace{1cm}

\end{document}